\documentclass[10pt,a4paper,twoside]{amsart} 
\usepackage[utf8]{inputenc}
\usepackage[english]{babel}
\usepackage{graphicx}
\usepackage[a4paper,width=150mm,top=25mm,bottom=25mm,bindingoffset=6mm]{geometry}
\usepackage{fancyhdr}

\usepackage[bookmarks=false]{hyperref}
\usepackage{frontespizio}

\usepackage{float}
\usepackage{amssymb}
\usepackage{latexsym}
\usepackage{amsthm}
\usepackage{amsmath}
\usepackage{amssymb}
\usepackage{dsfont}
\usepackage{mathrsfs}
 \usepackage{hyperref}
 \usepackage{multirow}
 \usepackage{booktabs}

\usepackage{tikz}
\usepackage{tikz-qtree}

\usepackage[skip=2pt]{caption}
\usepackage{tensor}

\makeatletter
\renewcommand{\tocsection}[3]{%
  \indentlabel{\@ifnotempty{#2}{\bfseries\ignorespaces#1 #2\quad}}\bfseries#3}
\renewcommand{\tocsubsection}[3]{%
  \indentlabel{\@ifnotempty{#2}{\ignorespaces#1 #2\quad}}#3}

\newcommand\@dotsep{4.5}
\def\@tocline#1#2#3#4#5#6#7{\relax
  \ifnum #1>\c@tocdepth 
  \else
    \par \addpenalty\@secpenalty\addvspace{#2}%
    \begingroup \hyphenpenalty\@M
    \@ifempty{#4}{%
      \@tempdima\csname r@tocindent\number#1\endcsname\relax
    }{%
      \@tempdima#4\relax
    }%
    \parindent\z@ \leftskip#3\relax \advance\leftskip\@tempdima\relax
    \rightskip\@pnumwidth plus1em \parfillskip-\@pnumwidth
    #5\leavevmode\hskip-\@tempdima{#6}\nobreak
    \leaders\hbox{$\m@th\mkern \@dotsep mu\hbox{.}\mkern \@dotsep mu$}\hfill
    \nobreak
    \hbox to\@pnumwidth{\@tocpagenum{\ifnum#1=1\bfseries\fi#7}}\par
    \nobreak
    \endgroup
  \fi}
\AtBeginDocument{%
\expandafter\renewcommand\csname r@tocindent0\endcsname{0pt}
}
\makeatother

\usepackage{lipsum}

\DeclareUnicodeCharacter{2009}{\,}

\newtheorem{thm}{Theorem}[section]
\newtheorem{defn}[thm]{Definition}
\newtheorem{lem}[thm]{Lemma}
\newtheorem{cor}[thm]{Corollary}

\newtheorem{prop}[thm]{Proposition}
\newtheorem{oss}[thm]{Remark}

\usepackage{amsaddr}

\usepackage{abstract}

\begin{document}

\title[Non existence of closed and bounded null geodesics in Kerr spacetimes]{Non existence of closed and bounded null geodesics in Kerr spacetimes}


\author{Giulio Sanzeni}


\address{Ruhr-Universit\"at Bochum,  Fakult\"at f\"ur Mathematik,  Universit\"atsstra\ss e 150,  44801,  Bochum, Germany}
\email{ giulio.sanzeni@rub.de}

\maketitle

\thanks{\noindent \textbf{Funding.} This research is funded by the Deutsche Forschungsgemeinschaft (DFG, German Research Foundation) – Project-ID 281071066 – TRR 191.}


\vspace{0.5cm}

 \begin{abstract}
The Kerr-star spacetime is the extension over the horizons and in the negative radial region of the slowly rotating Kerr black hole.  It is known that below the inner horizon,  there exist both timelike and null (lightlike) closed curves.  Nevertheless,  we prove that null geodesics can be neither closed nor even contained in a compact subset of the Kerr-star spacetime. 
\end{abstract}

\setcounter{tocdepth}{2}
\tableofcontents
\setcounter{tocdepth}{3}

\section{\textbf{Introduction}}

\subsection{The Kerr black hole}

Kerr spacetimes are solutions of Einstein’s vacuum field equations found by R. P.  Kerr \cite{Kerr-paper} which are stationary,  axisymmetric and asymptotically flat.  The solutions depend on two parameters: $M$,  the \textit{mass parameter},  and $a$,  the \textit{rotation parameter} (angular momentum per unit mass).  The static spherically symmetric Schwarzschild solution \cite{Schw_paper} is obtained from the Kerr solution by setting $a$ to zero.  Slowly rotating ($|a|<M$) Kerr spacetimes model the gravitational field of a rotating black hole (BH),  at least sufficiently far away from it.  By conservation of the angular momentum,  any astrophysical body,  such as a black hole,  is supposed to spin in the universe.  The recent image of the supermassive black hole at the center of the galaxy M87 obtained by the Event Horizon Telescope Collaboration \cite{BH_first_picture},  which is the first direct detection of a BH, is consistent with the shadow \cite{Falcke_shadow} predicted using the Kerr model.  Another empirical evidence is the gravitational waves signal detected by the LIGO interferometers \cite{grav_waves_paper}: the decay of the waveform agrees with the damped oscillations of a BH relaxing to a final Kerr configuration.  For these reasons,  the Kerr solution is a physically meaningful spacetime model to analyse.  

However,  the natural analytic extension in the negative radial region of the slowly rotating Kerr solution exhibits causality violations for all non-zero values of $a$: both timelike and null closed curves are present.  This fact,  first noticed by Carter \cite{Carter_causality},  is in sharp contrast with the Schwarzschild solution for which the maximal analytic extension is globally hyperbolic \cite{Landsman_book}.  Whether or not causality violating spacetimes can be considered physically reasonable is an open problem.  It is expected though that closed causal (timelike or lightlike) geodesics raise more conceptual problems than closed causal curves,  since causal curves correspond to orbits of accelerated bodies while geodesics to orbits of light or free-falling massive particles.  For example,  G\"odel's rotating dust spacetime \cite{Godel} does not contain any closed causal geodesics \cite{Kundt,  Chandr_Wright,  Nolan_godel} and the acceleration of all closed causal curves in it must be huge \cite{Nolan_godel}.  Vacuum spacetimes with closed timelike curves but no closed null geodesics were constructed by Li \cite{Li} and Low \cite{Low_1995}.  The (non-)existence of closed null geodesics is also related to the \textit{chronology protection conjecture} stated by Hawking in \cite{ChronogyProtection}. 
Thus,  our result suggests that the extension of the Kerr solution is not entirely physically unreasonable.  A natural further step in this direction would be to investigate the (non-)existence of closed timelike geodesics in Kerr spacetimes,  which still remains an open problem.

\subsection{Results}\label{subsection Result}
Given a spacetime $\big(\mathcal{M},\mathbf{g}\big)$,  \textit{i.e.} a time-oriented connected Lorentzian manifold,  and a geodesic curve $\gamma:I=[a,b]\rightarrow \mathcal{M}$,  we say that $\gamma$ is a \textit{closed geodesic}  if $\gamma(a)=\gamma(b)$ and $\gamma'(a)=\lambda\gamma'(b)\neq 0$,  for some real number $\lambda\neq 0$.  $\gamma$ is called \textit{bounded} is its image is contained in a compact set of $\mathcal{M}$.

The purpose of this paper is to prove the non existence of closed and bounded null (lightlike) geodesics in the Kerr-star extension of the slowly rotating ($|a|<M$) Kerr black hole,  described in detail in \S \ref{definiton of Kerr}. 

\begin{thm}\label{main theorem}
Let $K^*$ be the Kerr-star spacetime.  Then there are no closed null (lightlike) geodesics in $K^*$.
\end{thm}
Combining the proof of Thm.  \ref{main theorem} with additional geometrical arguments,  we are able to prove a more general result about null geodesics in the Kerr-star spacetime:

\begin{thm}\label{second theorem}
Let $K^*$ be the Kerr-star spacetime.  Then there are no bounded null (lightlike) geodesics in $K^*$.
\end{thm}

\subsection{Space of null geodesics}

The original motivation which led us to investigate the existence of closed null geodesics in the Kerr spacetime is the study of the space of null (lightlike) geodesics of this spacetime.  
 Penrose was the first to suggest the importance of the study of the space of null geodesics \cite{Penrose_twistor,  Penrose_book}.  First results in this context were obtained by Low \cite{Low_1989}.  He proved that if the spacetime is globally hyperbolic (we refer to \cite{Beem_book} for causal theory definitions),  then its space of null geodesics is contactomorphic to the spherical cotangent bundle of any Cauchy hypersurface of the spacetime,  as explained in \cite{LOW_spherical_cotangent,Nemirovski_legendrian}.  For this reason,  in the case of globally hyperbolic spacetimes,  causality can be described in terms of the geometry of the emerging contact manifold \cite{Nemirovski_2, Chernov_2016}.  In the case of causally simple spacetimes,  a sufficient condition to get a smooth contact manifold for the space of null geodesics  is the existence of a conformal open embedding of the spacetime into a globally hyperbolic one \cite{Hedicke_Suhr}.   Furthermore,  in \cite{Low_1989} it is proven that if $\mathcal{M}$ is a strongly causal spacetime,  then its space of null geodesics inherits a smooth structure from the cotangent bundle of $\mathcal{M}$,  but is not necessarily Hausdorff.  If the spacetime is not causal,  like Kerr,  there are no general results about its space of null geodesics,  except for Zollfrei spacetimes \cite{Zollfrei_book,SUHR_2013},  in which all the null geodesics are closed and the space of null geodesics is well understood.  From the study of null geodesic orbits  we hope to obtain insights into the structure of the space of null geodesics  of Kerr spacetimes.

\subsection{Geodesic motion in Kerr spacetimes}

Thanks to the existence of three obvious constants of motion (the energy,  associated to the timelike Killing vector field,  the angular momentum,  associated to the spacelike Killing vector field,  and  the Lorentzian energy of the geodesic),  geodesic motion can be solved in some special submanifolds,  since in such case three constants of motion are sufficient to completely integrate the system.  First Boyer and Price 
\cite{Boyer_Price_1965},  then Boyer and Lindquist 
\cite{Boyer-Lindquist_paper},  and then de Felice 
\cite{deFelice_1968} studied geodesic motion in the equatorial hyperplane $Eq=\{\theta=\pi/2\}$ in Kerr spacetime 
\cite{Kerr-paper}.  For the same reason,  Carter 
\cite{Carter_1966_Axis} was able to study orbits in the symmetry axis $A=\{\theta=0,\pi\}$.  Bounded orbits,  namely geodesics which run over a finite radial interval,  were studied by Wilkins 
\cite{Wilkins}.  After the maximal analytic extension of the Kerr metric by Boyer and Lindquist 
\cite{Boyer-Lindquist_paper},  Carter 
\cite{Carter_causality}  found a fourth constant of motion,  the Carter constant,  which completed the integrability of the geodesic flow (see for instance \cite{frolov_book}) and allowed the study of geodesics in full generality.  The most extensive and complete references about geodesic motion in Kerr spacetimes are probably the text-books by Chandrasekhar \cite{Chandrasekhar} and O'Neill \cite{KBH_book}.


\subsection{Organization of the paper}

In \S \ref{definiton of Kerr},  we introduce the Kerr metric and  discuss the definition and properties of the Kerr-star spacetime.  In \S \ref{study of geodesic equations} we recall the set of first order differential equations satisfied by geodesic orbits.  In \S \ref{section: properites of null geodesics},  we study the properties of null geodesics required to prove the main theorem.  In \S\ref{section: main theorem},  we give the proof of Thm.  \ref{main theorem} split into several cases.  The overall structure of the proof is detailed in \ref{strategy of the proof},  \ref{steps of other cases} and Fig.  \ref{figure steps of proof}.  In \S \ref{section unbounded},  we show that the null geodesics approaching a horizon and the incomplete ones cannot be bounded respectively in Prop.  \ref{prop unbounded geodesics approaching horizons} and in Cor.  \ref{incomplete are unbounded}. In \S \ref{proof second theorem} we give the proof of Thm.  \ref{second theorem}.

\vspace{0.5cm}

\normalsize {\thanks{ \noindent \textbf{Acknownledgments.} I would like to thank my PhD supervisors S.  Nemirovski and S.  Suhr for many fruitful discussions and precious advices.  I am also grateful to Liang Jin for the interesting and helpful conversations. }

\section{\textbf{ The Kerr-star spacetime}}\label{definiton of Kerr}

Consider $\mathbb{R}^2\times S^2$ with coordinates $(t,r)\in\mathbb{R}^2$ and $(\theta,\phi)\in S^2$.  Fix two real numbers $a\in\mathbb{R}\setminus \{0\}$,  $M\in \mathbb{R}_{>0}$  and define the functions 
\[
 \rho(r,\theta):= \sqrt{r^2+a^2\cos^2\theta}
\]
and 
\[
\Delta(r):=r^2-2Mr+a^2.
\]

We study the case $|a|<M$ called \textit{slow Kerr},  for which $\Delta(r)$ has two positive roots
\begin{align*}
r_{\pm}=M\pm \sqrt{M^2-a^2}>0
\end{align*}
and define two sets 
\begin{itemize}
\item[(1)] the \textit{horizons} $\mathscr{H}:=\{\Delta(r)=0\}=\{r=r_{\pm}\}:=\mathscr{H}_{-}\,\sqcup \mathscr{H}_{+}$, 
\item[(2)] the \textit{ring singularity} $\Sigma:=\{\rho(r,\theta)=0\}=\{r=0,\;\theta=\pi/2\}$.
\end{itemize}

The {\it Kerr metric}  \cite{Kerr-paper} in {\it Boyer--Lindquist coordinates} is

\begin{align}\label{kerr metric}
\mathbf{g}=-dt\otimes dt + \frac{2Mr}{\rho^2(r,\theta)}(dt-a\sin^2\theta\; d\phi)^2 + \frac{\rho^2(r,\theta)}{\Delta(r)}dr\otimes dr + a^2\sin^4(\theta) d\phi\otimes d\phi  + \rho^2(r,\theta) d\sigma^2  ,
\end{align}
where $d\sigma^2=d\theta\otimes d\theta + \sin^2\theta d\phi\otimes d\phi$ is the $2$-dimensional  (Riemannian) metric of constant unit curvature on the unit sphere $S^2\subset\mathbb{R}^3$ written in spherical coordinates. 

\begin{oss}
The components of $\mathbf{g}$ in Boyer--Lindquist coordinates can be read off the common expression

\begin{align}
\mathbf{g}=-&\bigg(1-\frac{2 M r}{\rho^2(r,\theta)} \bigg)\:dt\otimes dt-\frac{4Mar\sin^2\theta}{\rho^2(r,\theta)}\: dt\otimes d\phi + \nonumber\\+ &\bigg(r^2+a^2+\frac{2Mra^2\sin^2\theta}{\rho^2(r,\theta)} \bigg)\sin^2\theta\:  d\phi\otimes d\phi + \frac{\rho^2(r,\theta)}{\Delta(r)}\: dr\otimes dr + \rho^2(r,\theta)\: d\theta\otimes d\theta.
\end{align}

Nevertheless this last expression does not cover the subsets $\{\theta=0,\pi\}$.

\end{oss}

\begin{lem}
The metric \eqref{kerr metric} is a Lorentzian metric on $\mathbb{R}^2\times S^2\setminus (\Sigma\,\cup \mathscr{H})$.
\end{lem}

The sets on which the Boyer--Lindquist coordinates or the metric tensor fail are:

\begin{itemize}
\item the \textit{horizons} $\mathscr{H}=\{\Delta(r)=0\}=\{r=r_{\pm}\}=\mathscr{H}_{-}\sqcup\mathscr{H}_{+}$;
\item the \textit{ring singularity} $\Sigma=\{\rho(r,\theta)=0\}=\{r=0,\;\theta=\pi/2\}$.
\end{itemize}

In order to extend the metric tensor to the horizons,  one has to introduce a new set of coordinates.  No change of coordinates can be found in order to extend the metric across the ring singularity.  For a detailed study of the nature of the ring singularity,  see for instance \cite{Chrusc_singularity}.

\begin{defn}
The subsets 
\[
\textrm{I}:=\{r>r_{+}\},\;  \textrm{II}:=\{r_{-}<r<r_{+}\},\; \textrm{III}:=\{r<r_{-}\}\subset \{(t,r)\in\mathbb{R}^2,\; (\theta,\phi)\in S^2 \}\setminus (\Sigma\,\cup \mathscr{H})
\] 
are called the Boyer--Lindquist (BL) blocks.
\end{defn}

\begin{oss}
The BL blocks I,  II and III are the connected components of $\mathbb{R}^2\times S^2\setminus (\Sigma\,\cup \mathscr{H})$.  Each block with the restriction of the metric tensor \eqref{kerr metric} is a connected Lorentzian $4$-manifold.  To get spacetimes,  one has to choose a time orientation on each block.
\end{oss}

\begin{figure}[H]
\centering
\includegraphics[scale=0.4]{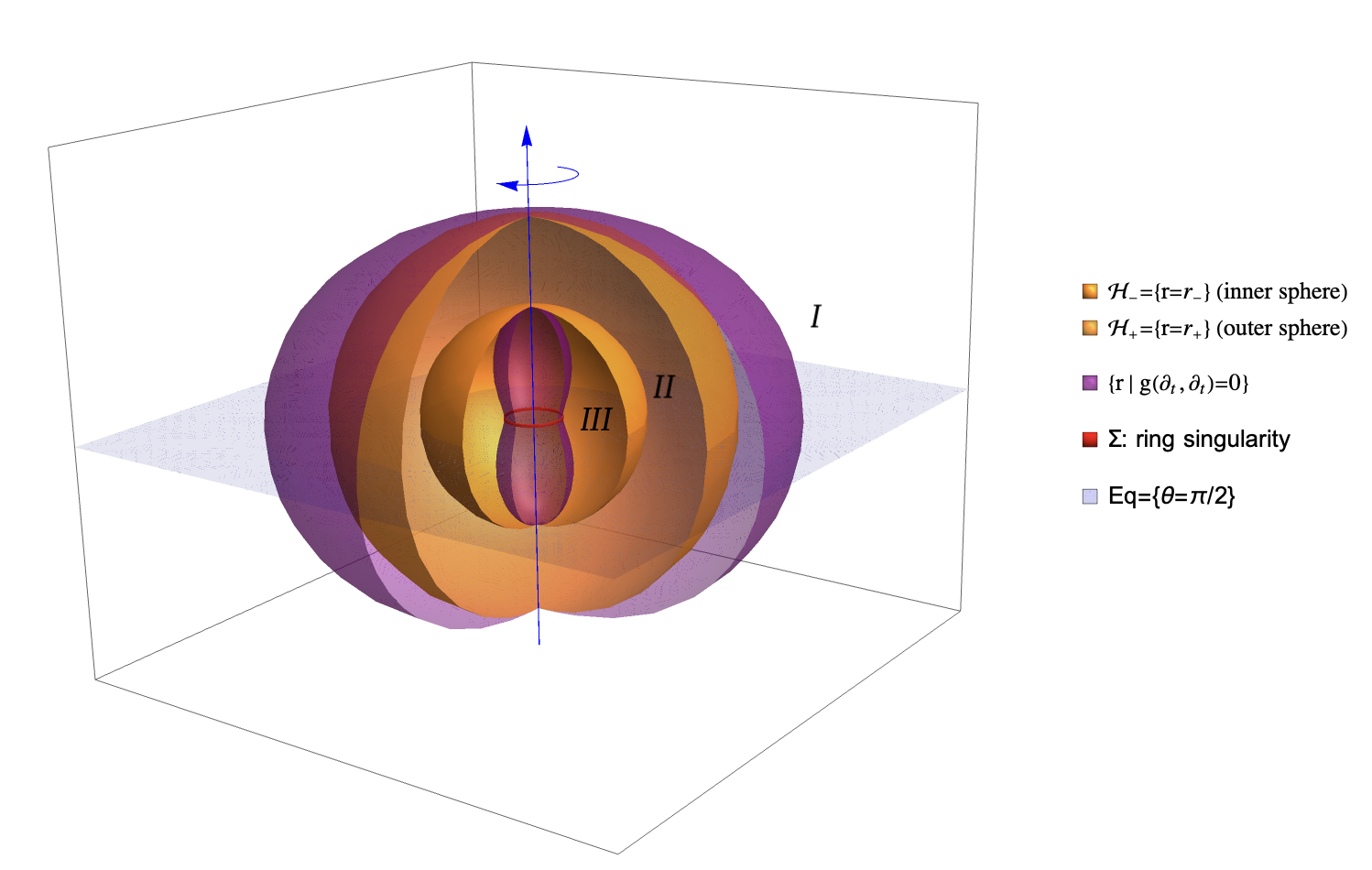} 
\caption{This picture shows a $t$-slice $\{t\}\times\mathbb{R} \times S^2$,  with the radius drawn as $e^r$,  so that $r=-\infty$ is at the center of the figure.  The \textit{Ergoregion} $\{\mathbf{g}(\partial_t,\partial_t)>0\}$ (at fixed time $t$) is  the region between the purple ellipsoids in which $\partial_t$ becomes spacelike.}
\vspace*{-5mm}
\end{figure}

\vspace{1cm}

\subsection{Time orientation of BL blocks }

We define a future time-orientation of block  I using the gradient timelike vector field $-\nabla t$.  Indeed,  the hypersurfaces $\{t=\textrm{const}\}$ are spacelike in block I.  Notice that the coordinate vector field $\partial_t$ is timelike future-directed for $r\gg r_+$ on block I, since $\mathbf{g}(-\nabla t,\partial_t)=-1$.\\

We define a time-orientation of block II by declaring the vector field $-\partial_r$,  which is timelike in II,  to be future-oriented.\\

We define a time-orientation of block III by declaring the vector field $V:=(r^2+a^2)\partial_t + a\partial_\phi$,  which is timelike in III,  to be future-oriented.\\

With this choice of time orientations,  each block is a \textit{spacetime},  \textit{i.e.} a connected time-oriented Lorentzian $4$-manifold.

\subsection{Kerr spacetimes}

\begin{defn}\label{definition kerr spacetime}
A {\it Kerr spacetime} is an analytic spacetime $(K,g_K)$ such that
\begin{enumerate}
\item[(1)]  there exists a family of open disjoint isometric embeddings $\Phi_i \colon \mathcal{B}_i\hookrightarrow K$ $(i\in \mathbb{N})$ of BL blocks $(\mathcal{B}_i,
\mathbf{g}_K|_{\mathcal{B}_i})$, such that $\cup_{i\in\mathbb{N}} \Phi_i(\mathcal{B}_i)$ is dense in $K$;

\item[(2)] there are analytic functions $r$ and $C$ on K  such that their restriction on each $\Phi_i(\mathcal{B}_i)$ of condition $(1)$ is $\Phi_i$-related to the Boyer--Lindquist functions $r$ and $C=\cos\theta$ on $\mathcal{B}_i$;

\item[(3)]  there is an isometry $\epsilon: K\rightarrow K$ called the \textit{equatorial isometry} whose restrictions to each BL block sends $\theta$ to $\pi-\theta$,  leaving the  other coordinates unchanged;

\item[(4)] there are Killing vector fields $\tilde{\partial_t}$ and $\tilde{\partial_\phi}$ on K that restrict  to the Boyer--Lindquist coordinate vector fields $\partial_t$ and $\partial_\phi$ on each BL block.
\end{enumerate}
\end{defn}

\begin{oss}
With abuse of notation,  we identify each block $\mathcal{B}_i$ with its image via the isometric embedding $\Phi_i(\mathcal{B}_i)\subset K$.
\end{oss}

\begin{lem}
Each time-oriented BL block is a Kerr spacetime.
\end{lem}

\begin{defn}\label{abstract def axis and Eq}
In a Kerr spacetime K,  on any BL block $\mathcal{B}_i$
\begin{enumerate}
\item the \textit{axis} $A=\{\theta=0,\pi\}$ is the set of zeroes of the Killing vector field $\tilde{\partial_\phi}$ as in $(4)$ of Def. \ref{definition kerr spacetime};
\item the \textit{equatorial hyperplane} $Eq=\{\theta=\pi/2\}$ is the set of fixed points of the equatorial isometry $\epsilon$ as in $(3)$ of Def.  \ref{definition kerr spacetime}.
\end{enumerate}
\end{defn}

\subsection{The Kerr-star spacetime}

\begin{defn}
On each BL block,  we define the \textit{Kerr-star coordinate} functions:
\begin{align}
t^*:=t+\mathcal{T}(r)\in\mathbb{R},\hspace{1cm} \phi^*:=\phi+\mathcal{A}(r)\in S^1,
\end{align}
with $d\mathcal{T}/dr:=(r^2+a^2)/\Delta(r)$ and $d\mathcal{A}/dr:=a/\Delta(r)$.
\end{defn}

\begin{lem}[\cite{KBH_book},  Lemma  $2.5.1$]
For each BL block $B$,  the map $\xi^*=(t^*,r,\theta,\phi^*):B\setminus A\rightarrow \xi^*(B)\subseteq\mathbb{R}^4$ is a coordinate system on $B\setminus A$,  where $A$ is the axis.  We call $\xi^*$ a \textit{Kerr-star} coordinate system.
\end{lem}

Because the  Kerr-star coordinate functions differ from BL coordinates only by additive functions of $r$,  the coordinate vector fields $\partial_t,\partial_\theta,\partial_\phi$ are the same in the two systems,  except that in $K^*$ they extend over the horizons.  Instead,  the new coordinate vector field $\partial^*_r=\partial_r-\Delta(r)^{-1}V$,  where $V$ is one of the canonical vector fields defined in Section \ref{study of geodesic equations}.  Note that if we use Kerr-star coordinates,  we get $\mathbf{g}(\partial^*_r,\partial^*_r)=0$,  \textit{i.e.} $\partial^*_r$ is a null vector field of $K^*$,  while in BL coordinates,  $\mathbf{g}(\partial_r,\partial_r)=\rho^2(r,\theta)/\Delta(r),$ which is singular when $\Delta(r)=0$.

\begin{lem}\label{Kerr-star metric}
The Kerr metric,  expressed in Kerr-star coordinates,  takes the form

\begin{align}
\mathbf{g}=&-\bigg(1-\frac{2 M r}{\rho^2(r,\theta)} \bigg) \:dt^*\otimes dt^* -\frac{4Mar\sin^2\theta}{\rho^2(r,\theta)}\: dt^*\otimes d\phi^*\,+\nonumber\\ 
&+ \bigg(r^2+a^2+\frac{2Mra^2\sin^2\theta}{\rho^2(r,\theta)} \bigg)\sin^2\theta\:   d\phi^*\otimes d\phi^* +2\: dt^*\otimes dr\, +\\
&-2a\sin^2\theta\: d\phi^*\otimes dr + \rho^2(r,\theta)\: d\theta\otimes d\theta.  \nonumber
\end{align}
\end{lem}

Now all coefficients in $\mathbf{g}$ are well defined on the horizons $\mathscr{H}=\{\Delta(r)=0\}$,  hence it is a well defined Lorentzian metric on  $\mathbb{R}^2\times S^2\setminus \Sigma$ and constitutes an analytic extension of \eqref{kerr metric} over $\mathscr{H}$.

\begin{defn}\label{Kerr-star spacetime}
The \textit{Kerr-star spacetime} is a Kerr spacetime as defined in \ref{definition kerr spacetime} given by the tuple $(K^*,\mathbf{g},o)$ with $K^*=\{(t^*,r)\in\mathbb{R}^2,\, (\theta,\phi^*)\in S^2\}\setminus\Sigma$,  $\mathbf{g}$ as in Lemma \ref{Kerr-star metric}   (extended over the axis) and $o$ is the future time-orientation induced by the null vector field $-\partial^*_r.$
\end{defn}

\begin{oss}
Note that the time-orientations on individual BL blocks agree with the ones defined for the Kerr-star spacetime:  $\mathbf{g}(-\partial^*_r,\partial_t)=-1<0$ on I,   $\mathbf{g}(-\partial^*_r,-\partial_r)=\mathbf{g}(\partial_r,\partial_r)=\rho^2(r,\theta)/\Delta(r)<0$ on II and $\mathbf{g}(-\partial^*_r,V)=\frac{1}{\Delta(r)}\mathbf{g}(V,V)=-\rho^2(r,\theta)<0$ on III.
\end{oss}

\subsection{Totally geodesic submanifolds of the Kerr-star spacetime}

\begin{lem}\label{A and Eq closed totally geod subman}
Let $K^*$ be the Kerr-star spacetime as in Def.  \ref{Kerr-star spacetime}.  The axis $A$ and the equatorial hyperplane $Eq$ of $K^*$ are closed totally geodesic submanifolds of $K^*$.
\end{lem}

\begin{prop}[\cite{KBH_book},  Prop.  $2.5.5$]\label{H is closed totally geod}
Let $K^*$ be the Kerr-star spacetime.  Then the horizon $\mathscr{H}$ is a closed totally geodesic null hypersurface,  with future hemicone on the $-\partial^*_r$ side.  Moreover,  the restriction of $V:=(r^2+a^2)\partial_t+a\partial_\phi$ (called canonical vector field in \S \ref{study of geodesic equations}) on $\mathscr{H}$ is the unique null vector field on $\mathscr{H}$ that is tangent to $\mathscr{H}$,  hence also normal to $\mathscr{H}$.  The integral curves of $V$ in $\mathscr{H}$ are null pregeodesics.
\end{prop}


\subsection{Causal region of the Kerr-star spacetime}

\begin{prop}[\cite{KBH_book},  Proposition $2.4.6$]\label{prop I and II casual}
The BL blocks I and II are causal.
\end{prop}

\begin{cor}\label{causal region of K^*}
Let $K^*$ be the Kerr-star spacetime.  Then the region\\
 I$\;\cup$ II$\; \cup\;  \{r=r_{\pm}\} =\{ t^*\in\mathbb{R},  r\in[r_{-},+\infty),  (\theta,\phi^*)\in S^2\}\setminus \Sigma\subset K^*$ is causal.
\end{cor}

\begin{proof}
Let $\gamma$ be a future pointing curve.  If $\gamma$ is entirely contained either in I or in II,  then by Prop.  \ref{prop I and II casual},   $\gamma$ cannot be closed.   If $\gamma$ is entirely contained in $\mathscr{H}=\{r=r_{\pm}\}$ (closed totally geodesic null hypersurface of $K^*$ by Prop.  \ref{H is closed totally geod}),  then by Lem.  $1.5.11$ of \cite{KBH_book},  except for restphotons,  all other curves are spacelike,  but restphotons are integral curves of $V|_\mathscr{H}=(r^2_\pm+a^2)\partial_t+a\partial_\phi$,  which cannot be closed.  Since the time orientation $-\partial^*_r$ is null and transverse to the null hypersurface $\mathscr{H}$,  the future directed curves always go in the direction of $-\partial^*_r$,  if they hit $\mathscr{H}$ transversally.  Henceforth,  if $\gamma$ starts in the BL block I (II),  crosses $\mathscr{H}_{+}$ ($\mathscr{H}_{-}$) transversally,  enters the block II (III),  then  $\gamma$ cannot re-intersect $\mathscr{H}_{+}$  from II to I ($\mathscr{H}_{-}$ from III to II).  The last possibility is the following: $\gamma$ starts in I (II),  becomes tangent to $\mathscr{H}_{+}$ ($\mathscr{H}_{-}$),   hence either lies forever on $\mathscr{H}_{+}$ ($\mathscr{H}_{-}$) or leaves it at some point.  In the first case,  $\gamma$ is obviously not closed,  while in the second,  it cannot be closed because it will necessarily have to enter the region $\{r<r_{+}\}$ ($\{r<r_{-}\}$),  according to the time orientation. 
\end{proof}

\section{\textbf{Geodesics in Kerr spacetimes}}\label{study of geodesic equations}

\subsection{Constants of motion}

Let $(K,\mathbf{g})$ be a Kerr spacetime as in Def.  \ref{definition kerr spacetime}.  Recall that there are two Killing vector fields $\tilde{\partial_t}$ and $\tilde{\partial_\phi}$ on $K$. 

\begin{defn}[\textit{Energy and angular momentum}] 
For a geodesic $\gamma$ of $(K,\mathbf{g})$,  the constants of motion 
\[
E=E(\gamma):=-\mathbf{g}(\gamma',\tilde{\partial_t})
\]
and 
\[
L=L(\gamma):=\mathbf{g}(\gamma',\tilde{\partial_\phi})
\]
are called its {\it energy} and its {\it angular momentum (around the axis of rotation of the black hole)}, respectively.
\end{defn}

\begin{defn}
For every BL block $\mathcal{B}_i$ define the \textit{canonical vector fields}
\[
V:=(r^2+a^2)\partial_t + a\partial_\phi\quad\text{ and }\quad W:=\partial_\phi + a \sin^2\theta\,\partial_t
\]
via the isometry $\Phi_i\colon  \mathcal{B}_i\hookrightarrow K$.  
\end{defn}

\begin{oss}
$V$ and $W$ are not Killing vectors.
\end{oss}

\begin{defn}\label{definitions of P and D}
Let $\gamma$ be a geodesic in $K$ with energy $E$ and angular momentum $L$.  Define the functions $\mathbb{P}$ and $\mathbb{D}$ along $\gamma$ by
\[
\mathbb{P}(r):=-\mathbf{g}(\gamma',V)=(r^2+a^2)E-La
\]
and 
\[
\mathbb{D}(\theta):=\mathbf{g}(\gamma',W)=L-Ea\sin^2\theta.
\]\\
\end{defn}

A geodesic in a Kerr spacetime has two additional constants of motions.  First,  there is the {\it Lorentian energy} $q:=\mathbf{g}(\gamma',\gamma')$,  which is always constant along every geodesic
in any pseudo-Riemannian manifold.  The second one is $K$,  which was first found by Carter in \cite{Carter_causality} using the separability of the Hamilton--Jacobi equation.  $K$ can be defined (see Ch.  $7$ in \cite{Chandrasekhar}) by

\[
K:=2\rho^2(r,\theta)\mathbf{g}(l,\gamma')\mathbf{g}(n,\gamma')+r^2q,
\]
where $l=\frac{1}{\Delta(r)}V+\partial_r$ and $n=\frac{1}{2\rho^2(r,\theta)}V-\frac{\Delta(r)}{2\rho^2(r,\theta)}\partial_r$.  See also \cite{Walker_Penrose} for a definition using a Killing tensor for the Kerr metric.  

\begin{defn}[\textit{Carter constant}]
On a Kerr spacetime,  the constant of motion
\[
Q:=K-(L-aE)^2\hspace{0.5cm}\textrm{or}\hspace{0.5cm} \mathcal{Q}:=Q/E^2\hspace{0.3cm} \textrm{if}\hspace{0.3cm}E\neq 0
\]
is called the Carter constant.  
\end{defn}

\subsection{Equations of motion}

\begin{prop}[\cite{KBH_book},  Proposition $4.1.5$,  Theorem $4.2.2$] \label{differential equations of geodessics}
Let $B$ be a BL block and $\gamma$ be a geodesic with initial position in $B\subset K$ and constants of motion $E,  L,  Q,  q$.   Then the components of $\gamma$ in the BL coordinates $(t,r,\theta,\phi)$ satisfy the following set of \textit{first} order differential equations

\begin{align}
\begin{cases}
 \rho^2(r,\theta)\phi'=\frac{\mathbb{D}(\theta)}{\sin^2\theta}+a\frac{\mathbb{P}(r)}{\Delta(r)} \\ \rho^2(r,\theta) t'= a\mathbb{D}(\theta) + (r^2+a^2)\frac{\mathbb{P}(r)}{\Delta(r)} \label{geodes diff equations}\\ \rho^4(r,\theta) r'^2 = R(r)\\  \rho^4(r,\theta) \theta'^2 = \Theta (\theta)   
 \end{cases}
\end{align}
where

\begin{align*}
R(r):=&\Delta(r)\left[(qr^2-K(E,L,Q)\right]+\mathbb{P}^2(r)= \\ 
=& (E^2+q)r^4 -2Mqr^3 + \mathfrak{X}(E,L,Q) r^2 + 2MK(E,L,Q)r - a^2 Q,\label{other form of R(r)}\\
\Theta(\theta):=& K(E,L,Q)+qa^2 \cos^2\theta -\frac{ \mathbb{D}(\theta)^2}{\sin^2\theta}= \\
=& Q + \cos^2\theta \left[ a^2(E^2+q)-L^2/\sin^2\theta\right],
\end{align*}
with
\[
\mathfrak{X}(E,L,Q):=a^2(E^2+q)-L^2-Q\text{, and } K(E,L,Q)=Q+(L-aE)^2.
\]

\end{prop}

\begin{oss}\label{non negativitivity of polynomials}
Since in the third and in the fourth differential equations of Prop.  \ref{differential equations of geodessics} the left-hand sides are clearly non-negative,  we see that the polynomials $R(r)$ and $\Theta(\theta)$ are non-negative along the geodesics.  Hence the geodesic motion can only happen in the $r,\theta$-region for which $R(r),\Theta(\theta)\geq 0$.  
\end{oss}

In order to study geodesics that cross the horizons 
\[
\mathscr{H}=\{\Delta(r)=0\}=\{r=r_{\pm}\},  
\]
it is necessary to introduce the Kerr-star coordinate system.  Note however that since the change of coordinates modifies only the $t$ and the $\phi$ coordinates and the $r,\theta$-differential equations do not involve $t$ and $\phi$,  the last two differential equations do extend over $\mathscr{H}$.  Observe also that the $r,\theta$- differential equations are not singular on $\mathscr{H}$,  while the $t,\phi$-differential equations are. \\

Notice that $\Theta(\theta)$ is also well-defined if the null geodesic crosses $A=\{\theta=0,\pi\}$.  Indeed,  $L=0$ (because  $\tilde{\partial_\phi}\equiv 0$ on $A$),  hence $\mathbb{D}(\theta)=-Ea\sin^2\theta$,  and then
  
\[ 
\Theta(\theta)=K(E,0,Q)-(-Ea\sin^2\theta)^2/\sin^2\theta=Q+a^2E^2-a^2E^2\sin^2\theta=Q+a^2E^2\cos^2\theta.\\
\]
Thus the $r,\theta$-differential equations can be used to study geodesics on the whole Kerr-star spacetime.

\begin{oss}
The system \eqref{geodes diff equations} is composed of first order differential equations,  while the geodesic equation is \textit{second} order.  There exist solutions of \eqref{geodes diff equations},  called \textit{singular},  which do not correspond to geodesics.  For example,  if $r_0\in\mathbb{R}$ is a multiplicity one zero of $r\mapsto R(r)$,  then $r_0$ solves the radial equation in \eqref{geodes diff equations},  since in this case $r'(s)=0$ for all $s$,  but we do not have a geodesic. 
\end{oss}

\subsection{Dynamics of geodesics}

The non-negativity of $R(r)$ and $\Theta(\theta)$ in the first order differential equations of motion \eqref{geodes diff equations} can be used to study the dynamics of the $r,\theta$-coordinates of the geodesics,  together with the next proposition.

\begin{prop}[\cite{KBH_book},  Corollary $4.3.8$]\label{initial conditions and zeroes}
Suppose $R(r_0)=0$.  Let $\gamma$ be a geodesic whose $r$-coordinate satisfies the initial conditions $r(s_0)=r_0$ and $r'(s_0)=0$.
\begin{enumerate}
\item If $r_0$ is a multiplicity one zero of $R(r)$,  \textit{i.e.} $R'(r_0)\neq 0$,  then $r_0$ is an $r$-turning point,  namely $r'(s)$ changes sign at $s_0$.
\item If $r_0$ is a higher order zero of $R(r)$,  \textit{i.e.} at least $R'(r_0)= 0$,  then $\gamma$ has constant $r(s)=r_0$.  
\end{enumerate}
Analogous results hold for $r$ and $R(r)$ replaced by $\theta$ and $\Theta(\theta)$.  
\end{prop}

\section{\textbf{Properties of null geodesics in Kerr spacetimes}}\label{section: properites of null geodesics}

\subsection{Principal geodesics}

Since the vector fields $V, W, \partial_r,\partial_\theta$ are mutually orthogonal,  the tangent vector to a geodesic $\gamma$ can be decomposed as $\gamma'=\gamma'_\Pi + \gamma'_\perp$ where $\Pi:= span \{ \partial_r,  V \}$ (timelike plane) and $\Pi^\perp:= span \{ \partial_\theta,  W \}$ (spacelike plane).


 


\begin{defn}
A Kerr geodesic $\gamma$ is said to be \textit{principal} if $\gamma' = \gamma'_\Pi$. 
\end{defn}

\begin{prop}[\cite{KBH_book},  Corollary $4.2.8(2)$] \label{eq principal nulls}
If $\gamma$ is a null geodesic,  then $K\geq 0$,  and\\ $K=0\Longleftrightarrow\;$ $\gamma$ is principal.
\end{prop}







\begin{defn}
A null geodesic is called a restphoton if it lies in $\mathcal{H}$.
\end{defn}

Restphotons are integral curves of $V|_{\mathcal{H}}$ by Prop.  \ref{H is closed totally geod}.


\begin{prop}\label{restphotons are unbounded}
Restphotons are unbounded in $K^*$.
\end{prop}
\begin{proof}
Let $\gamma_{\pm}$ be the integral curve of $V|_{\mathcal{H}_{\pm}}$.  By Lemma $3.3.1$ of \cite{KBH_book},  the time component of $\gamma_{\pm}$ is

\begin{align*}
t^*\big( \gamma_{\pm}(s) \big)=(r_{\pm}^2+a^2)\frac{\ln(k_{\pm}s)}{k_{\pm}},
\end{align*}
where $k_{\pm}=(r_{\pm}-r_{\mp})/2$ and $\gamma_{+}:(0,+\infty)\rightarrow K^*$ and $\gamma_{-}:(-\infty,0)\rightarrow K^*$.  Then 

\begin{align*}
\lim_{s\rightarrow\pm\infty} t^{*} (\gamma_{\pm}(s))=\pm\infty.
\end{align*}

\end{proof}




\begin{prop}[\cite{KBH_book},  Lemma $4.2.9(3,4)$]\label{E=L=K=0 prop}
For a null geodesic $\gamma$,
\begin{enumerate}
\item $K=L=0$ but $E\neq 0 \Leftrightarrow \gamma\in A\setminus\mathcal{H}$. 
\item $K=L=E=0 \Leftrightarrow$ $\gamma$ is a restphoton.
\end{enumerate}
\end{prop}

\subsection{Null geodesics with $Q<0$}

\begin{prop}\label{Q<0 condition}
Let $\gamma$ be a null geodesic with $Q<0$.  Then

\begin{enumerate}
\item $\gamma$ does not intersect $Eq=\{\theta=\pi/2\}$;
\item $a^2E^2> L^2$ and in particular $E\neq 0$.
\end{enumerate}
\end{prop}

\begin{proof}
If $\gamma\cap A=\emptyset$,  then from the $\theta$-equation of \eqref{geodes diff equations} we have

\begin{align*}
\cos^2\theta [L^2/\sin^2\theta-a^2E^2]=Q-\rho^4(r,\theta)\theta'^2<0.
\end{align*}
Hence $\cos^2\theta\neq 0$ and $L^2/\sin^2\theta-a^2E^2<0$,  hence $\gamma\cap Eq=\emptyset$ and $a^2E^2>L^2$,  so $E\neq 0$.

If $\gamma\cap A\neq\emptyset$,  then by Prop.  \ref{E=L=K=0 prop} $L=0$ and 

\begin{align*}
-a^2E^2\cos^2\theta =Q-\rho^4(r,\theta)\theta'^2<0.
\end{align*}
Therefore $\cos^2\theta\neq 0$ and $a^2E^2>0$,  so $\gamma\cap Eq=\emptyset$ and $E\neq 0$.
\end{proof}

\begin{oss}
If a geodesic $\gamma$ has $Q>0$,  then its $\theta$-motion is an oscillation around $\theta=\pi/2$,  so it cuts repeatedly through $Eq$ (see Propositions $4.5.4,  4.5.5$ in \cite{KBH_book}). 
\end{oss}

\begin{prop}\label{geod Q<0}
For $Q<0$ null geodesics,  $R(r)$ is convex and has either no roots or two negative roots. 
\end{prop}

\begin{proof}
First we show that $R(r)$ does not have zeroes in $r\geq 0$.  It is sufficient to show that the coefficient of $R(r)$ at $r^0$ is positive and every other coefficient is non-negative.  We know from Prop. \ref{differential equations of geodessics} that for null geodesics

\begin{align}
R(r)=E^2r^4 + \mathfrak{X} r^2 + 2MKr - a^2 Q.
\end{align} 

From Prop.  \ref{Q<0 condition} and Prop.  \ref{eq principal nulls},  we see that the coefficients are

\begin{itemize}
\item  $E^2>0$ at $r^4$;
\item $ \mathfrak{X}=a^2E^2-L^2-Q>-Q>0$ at $r^2$;
\item $K\geq 0$ at $r$;
\item $-a^2Q>0$ at $r^0$.
\end{itemize}

We now show that $R(r)$ is convex and hence has at most two real zeroes.  We have $R'=4E^2r^3+2\mathfrak{X}r+2MK$, hence $R''=12E^2r^2+2\mathfrak{X}$. Since $\mathfrak{X}>0$,  $R''$ never vanishes and $R''(0)=2\mathfrak{X}>0$.  Then $R''>0$ everywhere,  which means that $R$ is convex,  hence $R'$ has a unique zero.  This means that $R(r)$ has a unique critical point and hence it has either $0$ or  $2$ roots,  which have to be negative.

\begin{figure}[H]
\centering
\includegraphics[scale=0.7]{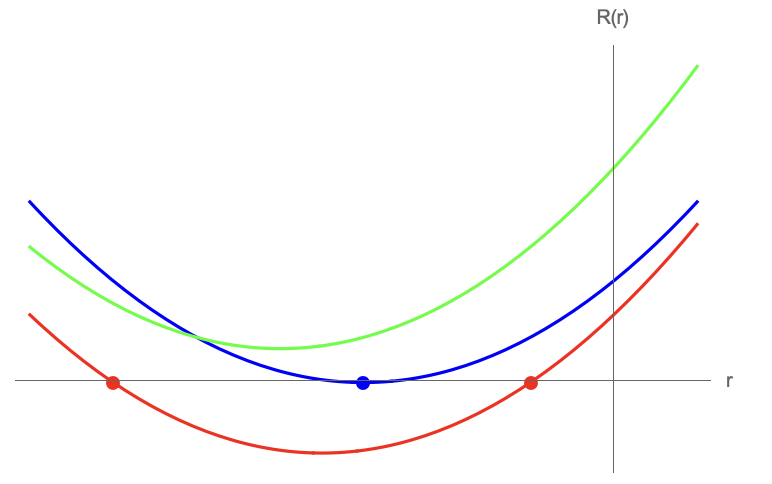} 
\caption{For $Q<0$ null geodesics,  $R(r)$ has either zero or two negative roots (multiplicity two is allowed). \\}
\vspace*{-5mm}
\end{figure}

\end{proof}

\section{\textbf{Proof of Theorem  \ref{main theorem}}} \label{section: main theorem}

\subsection{Strategy of the proof} \label{strategy of the proof}

Let $\gamma\colon I \to K^*$ be a closed null geodesic (CNG).
Since the radius function $r\colon K^*\to\mathbb{R}$ is everywhere smooth the composition $r\circ \gamma$ has at least two critical points $s_0<s_1$ in each period $[a,a+T)$, i.e.  
$(r\circ \gamma)'(s_0)=(r\circ \gamma)'(s_1)=0$. Since $\rho\colon K^*\to\mathbb{R}$  does not vanish on $K^*$ the differential equation for $r\circ \gamma$
\[
(\rho\circ \gamma)^4[(r\circ \gamma)']^2=R(r\circ \gamma)
\]
implies that $R(r\circ \gamma(s_{0,1}))=0$.  Because of the differential equation,  the geodesic motion must happen in the $r$-region on which $R(r\circ \gamma)\geq 0$.  Further since $R$ is a polynomial in $r$ we can distinguish two cases:
\begin{enumerate}

\item The zeros $r\circ \gamma(s_{0,1})$ of $R$ are simple, i.e. $dR/dr\neq 0$ at these points. Then $r\circ \gamma(s_{0,1})$ are turning points of $r\circ \gamma$, i.e. 
$(r\circ \gamma)'$ changes its sign at $s_0$ and $s_1$.  

\item One of the zeros $r\circ \gamma(s_{0})$ or $r\circ \gamma(s_{1})$ is a higher order zero of $R$. Then $r\circ \gamma$ is constant.
\end{enumerate}
Both the two facts follow from Proposition \ref{initial conditions and zeroes}.  

Most possible CNGs can be ruled out by comparing the location of the zeros of $R(r)$ with the following consequence of the causal structure of Kerr:

\begin{lem}\label{lemma about closed curves confined}
Let $\gamma\colon  I\to K^*$ be a closed null geodesic. Then $r\circ \gamma\subset \{r<r_-\}$.
\end{lem}

\begin{proof}
The region 
\[
\{r\geq r_{-}\}=\{ t^*\in\mathbb{R},  r\in[r_{-},+\infty),  (\theta,\phi^*)\in S^2\} \setminus \Sigma\subset K^*
\] 
is causal by Corollary  \ref{causal region of K^*} and closed null geodesics cannot intersect $\{r=r_{-}\}$ by Prop.  \ref{oss event horizon}. 
\end{proof}

\begin{prop}\label{prop spacelike foliation}
There are no closed null geodesics strictly contained in the region $\{0<r<r_{-}\}$.
\end{prop}

\begin{proof}
First we claim that the hypersurfaces $\mathcal{N}_t:=\{t=\textrm{const}\}\cap\{0<r<r_{-}\}\subset K^*$ are spacelike.  Indeed,  if $p\in\mathcal{N}_t\setminus A$,  where $A=\{\theta=0,\pi\}$,  then $T_p\mathcal{N}_t$ is spanned by $\partial_r,\partial_\theta,\partial_\phi$ which are spacelike and orthogonal to each other.  If $p\in A\subset\mathcal{N}_t$,  then $p=(t,r,q)$ with $q=(0,0,\pm 1)\in S^2\subset \mathbb{R}^3$,  and we may replace $\partial_\theta,\partial_\phi$ by any basis of $T_qS^2$.  Suppose by contradiction that there exist a CNG $\gamma$ in $\{0<r<r_{-}\}$.  Then there would exist $t_0,s_0$ such that $\gamma'(s_0)\in T_{\gamma(s_0)}\mathcal{N}_{t_0}$.  This is a contradiction since $\gamma'(s_0)$ is null.
\end{proof}

\hspace{1cm}

There are two cases in which we need additional arguments:

\begin{enumerate}
\item to exclude CNGs with $E=0,\; K(0,L,Q)=0,\;  L\neq 0$,  we use that the $t$-coordinate is monotonically increasing/decreasing;

\item to exclude CNGs with $Q<0$ and $r=\textrm{const}<0$,  we show that the $\theta$-coordinate of such a geodesic is periodic whereas the $t$-coordinate is quasi-periodic with a non-zero increment,  see \ref{case Q<0}.
\end{enumerate}

\subsection{Null geodesics in the horizons and axis}

First we rule out CNGs entirely contained in the axis $A=\{\theta=0,\pi\}$ and CNGs entirely contained/intersecting  the horizon $\mathscr{H}=\{r=r_{\pm}\}$.

\subsubsection*{The case of the horizon $\mathscr{H}=\{r=r_{\pm}\}$}

\begin{prop}\label{oss event horizon}
There are no CNGs intersecting $\mathscr{H}=\{r=r_{\pm}\}$.
\end{prop}

\begin{proof}

On the submanifolds $W:=\{r=\textrm{const}\}$, we have $\textrm{det}\;g^*|_W=-\rho^2(r,\theta)\Delta(r)\sin^2\theta$,  where $g^*$ is the Kerr-star coordinates expression of $\mathbf{g}$.  Therefore the metric $g^*$ degenerates on the tangent spaces to $\mathscr{H}=\{\Delta(r)=0\}=\{r=r_{\pm}\}$.  Therefore $T_p \mathscr{H}$ is a null subspace of $T_p K^*$ for every $p\in K^*$,  i.e.  the submanifolds $\mathscr{H}=\{\Delta(r)=0\}$ are null hypersurfaces.   Then,  by Lem.  $1.5.11$ of \cite{KBH_book},  every vector in $T_p K^*$ is spacelike,  except for the intersection of $T_p \mathscr{H}$ with the null cone of $T_p K^*$.  Note that the vector field $V=(r^2+a^2)\partial_t+a\partial_\phi$ is tangent to $\mathscr{H}$ and we have $\mathbf{g}(V,V)=-\Delta(r)\rho^2(r,\theta)$, i.e. $V$ is null along $\mathscr{H}$.  Hence $V|_\mathscr{H}$ generates the unique null tangent line to $H$. Further note that no flowline of $V$ closes.  By Prop.  \ref{H is closed totally geod},  the flowlines of $V|_\mathscr{H}$ are null pregeodesics,  hence the null geodesics tangent to $\mathscr{H}$ do not close. 

It remains to consider null geodesics intersecting $\mathscr{H}$ transversally.  Since each connected component $\{r=r_\pm\}$ of $\mathscr{H}$ is an orientable hypersurface separating the orientable manifold $K^*$,  every closed curve transversal to $H$ has to intersect $\{r=r_\pm\}$ an even number of times.  Further since $K^*$ is time-oriented by $-\partial_r^*$, all tangent vectors to a null geodesics transversal to $H$ have to lie on one side of $H$.  Therefore a null geodesic transversal to $H$ can intersect each connected component $\{r=r_\pm\}$ only once.  This shows that no null geodesic transversal to $\mathscr{H}$ can close.
\end{proof}

\subsubsection*{The case of the axis $A=\{\theta=0,\pi\}$}

\begin{prop} \label{no closed geodesics in axis}
There are no CNGs which are tangent at some point to $A=\{\theta=0,\pi\}$.  In particular,  there are no CNGs entirely contained in $A$.
\end{prop}

\begin{proof}
First of all,  $A=\{\theta=0,\pi\}$ is a $2$-dimensional closed totally geodesic submanifold by Lem.  \ref{A and Eq closed totally geod subman}.    Hence if a geodesic $\gamma$ is tangent to $A$ at some point,   it will always lie on $A$.  By Prop.  \ref{E=L=K=0 prop},  if $\gamma\in A$,  then $L=K=0$.  Hence there are two possible cases:  if $E=0$,  by Prop.  \ref{E=L=K=0 prop},  then $\gamma$ is a restphoton,  \textit{i.e.} an integral curve of $V|_\mathscr{H}=(r_{\pm}^2+a^2)\partial_t+a\partial_\phi$,  which is not closed.  If $E\neq 0$,  then using \eqref{geodes diff equations},  we have

\begin{align*}
R(r)=E^2(r^2+a^2)^2>0.
\end{align*}
So $R(r)$ has no zeroes and therefore the geodesics cannot be closed.

\end{proof}

\subsection{Steps of the proof for other cases}\label{steps of other cases}

The proof splits into two main cases  $E=0$ and $E\neq 0$.\\  

\noindent If $E=0$ (\ref{section $E=0$}),  we analyse two subcases $K(0,L,Q)=0$ (\ref{subcase E=0,K=0}) and $K(0,L,Q)>0$ (\ref{subcase E=0,K>0}). 

\noindent If $E\neq 0$ (\ref{section $E div 0$}),  we analyse three subcases $Q=0$ (\ref{case Q=0}),  $Q>0$ (\ref{case Q>0}) and $Q<0$ (\ref{case Q<0}).\\

\begin{oss}
The only case which requires a detailed analysis of the differential equations is the case $E\neq 0$,  $Q<0$ (see \ref{case Q<0}).
\end{oss}

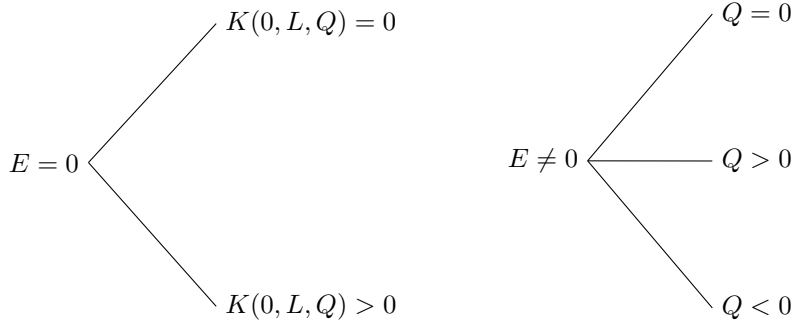
\begin{figure}[h]
\centering
\begin{tikzpicture}[grow=right]
\tikzset{level distance=100pt,sibling distance=90pt}
 \tikzset{frontier/.style={distance from root=200pt}}

\Tree 
               [.$E=0$ [.$K(0,L,Q)>0$ ] 
                       [.$K(0,L,Q)=0$ ] 
                       ]   
                         
\end{tikzpicture}
\hspace{1cm}
\begin{tikzpicture}[grow=right]
\tikzset{level distance=80pt,sibling distance=40pt}
 \tikzset{frontier/.style={distance from root=300pt}}

\Tree 
               [.$E\neq 0$  [.$Q<0$ ] 
               		              [.$Q>0$ ] 
                                    [.$Q=0$ ] 
                     ]

\end{tikzpicture}
\vspace*{5mm}
\caption{All the geodesic types which have to be studied.}
\label{figure steps of proof}
\vspace*{-5mm}
\end{figure}

\subsection{Case $E=0$}\label{section $E=0$}

\vspace{1cm}

From Prop.  \ref{differential equations of geodessics},  for null $(q=0)$ geodesics we have

\begin{align}
\label{R equation E=0}R(r)&=\mathfrak{X}(0,L,Q)r^2+2MK(0,L,Q)r-a^2Q\geq 0,\\
\label{theta equation E=0}\Theta(\theta)&=Q-\frac{\cos^2\theta}{\sin^2\theta}L^2\geq 0,
\end{align}
with $\mathfrak{X}(0,L,Q)=-(L^2+Q)$ and $K(0,L,Q)=L^2+Q$,  \textit{i.e.}  $\mathfrak{X}=-K$.  Notice that we must have $Q\geq 0$ by \eqref{theta equation E=0},  hence $K(0,L,Q)=L^2+Q\geq 0$,  as already known from Prop.  \ref{eq principal nulls}.
\vspace{1cm}

\subsubsection{Subcase $K(0,L,Q)=0$}\label{subcase E=0,K=0}

We have

\begin{align}
R(r)=-a^2Q\geq 0.
\end{align}
Since $Q\geq 0$,  we must then have $Q=0$.  From $\rho(r,\theta)^4r'^2=R(r)\equiv 0$,  we have $r(s)=r=\textrm{const}$ for every $s$.  Let us consider the cases $L\neq 0$ and $L=0$.

Assume $L\neq 0$.  From \eqref{theta equation E=0} we have $\theta(s)=\pi/2$ for every $s$.  Therefore the constant $r$-coordinate cannot be $r=0$,  as otherwise the geodesic would lie on the ring singularity.  By Lemma \ref{lemma about closed curves confined},  we can suppose $r=\textrm{const}<r_{-}$ and hence use the $t$-differential equation of \eqref{geodes diff equations} which becomes

\begin{align*}
\rho^2(r,\pi/2)t'=-\frac{2aMLr}{\Delta(r)}\neq 0.
\end{align*}
Hence $t'(s)\neq 0$ for every $s$ and so $t(s)$ must be monotonically increasing/decreasing,  therefore the geodesic cannot be closed.  

Assume now $L=0$.  From Prop.  \ref{E=L=K=0 prop},  a null geodesic with $L=E=K=0$ is a restphoton,  hence it cannot be closed by Prop.  \ref{oss event horizon}.


\subsubsection{Subcase $K(0,L,Q)>0$}\label{subcase E=0,K>0}

Since $|a|<M$ and $L^2+Q>Q\geq 0$,  the discriminant of \eqref{R equation E=0} is $\textrm{dis}=4M^2(L^2+Q)^2-4a^2Q(L^2+Q)>0$.  Therefore $R(r)$ has two roots given by

\begin{align*}
M\pm\sqrt{M^2-\frac{a^2Q}{L^2+Q}}.
\end{align*}

\subsubsection*{If $L\neq0$,} then we have (see Figures \ref{R(r) E=0,  L^2+Q dif 0,  Q>0},  \ref{R(r) E=0,  L^2+Q dif 0,  Q=0})

\[
M+\sqrt{M^2-\frac{a^2Q}{L^2+Q}}>r_{+}=M+\sqrt{M^2-a^2}.
\]

\begin{figure}[H]
\centering
\includegraphics[scale=0.6]{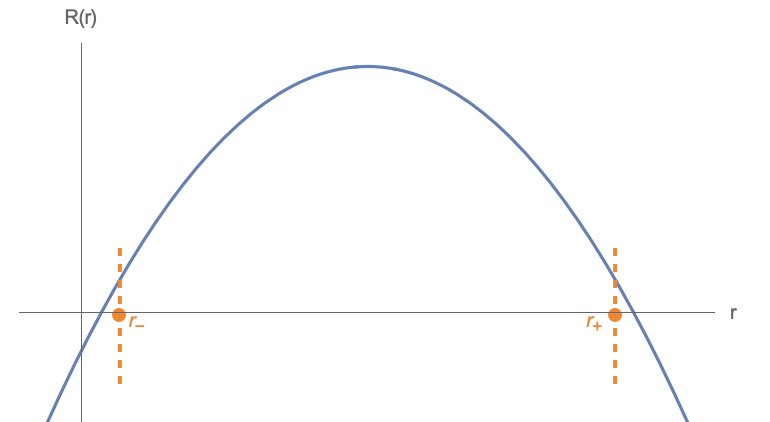} 
\caption{Plot of $R(r)$ in the case $E=0,\; L^2+Q\neq 0,\; Q>0$\\ with $a=3,M=6,L=2,Q=4$.}
\label{R(r) E=0,  L^2+Q dif 0,  Q>0}
\vspace*{-5mm}
\end{figure}

\begin{figure}[H]
\centering
\includegraphics[scale=0.6]{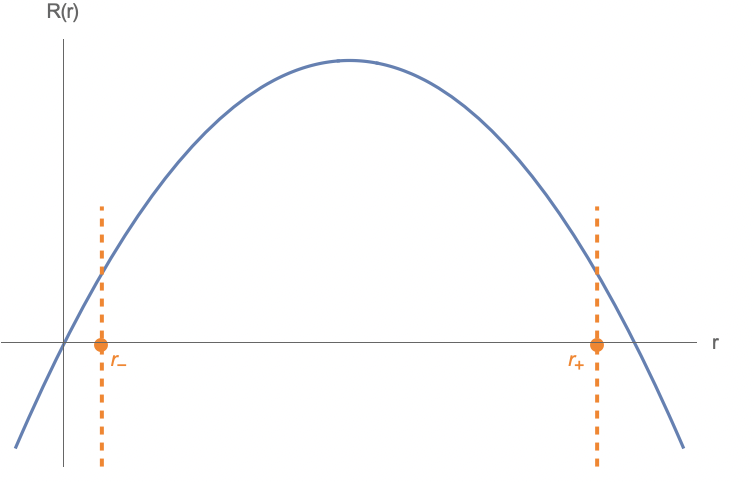} 
\caption{Plot of $R(r)$ in the case $E=0,\; L^2+Q\neq 0,\; Q=0$\\ with $a=3,M=6,L=2,Q=0$.}
\label{R(r) E=0,  L^2+Q dif 0,  Q=0}
\vspace*{-5mm}
\end{figure}

\hspace{3cm}

So if $L\neq 0$,  the geodesics will have to cross $\mathscr{H}$ since one of the two zeros is bigger than $r_{+}$ and the other is smaller than $r_{-}$,   which is impossible for a CNG by Prop.  \ref{oss event horizon}. \\

\subsubsection*{If $L=0$,} then $Q>0$,  hence $R''(r)=-2Q<0$.  Moreover $R(r)=-Q\Delta(r)$ because $E=L=0$. Therefore the two (multiplicity one) roots are $r_{-}$ and $r_{+}$ (see Fig.  \ref{R(r) E=0,  L=0}).

\begin{figure}[H]
\centering
\includegraphics[scale=0.6]{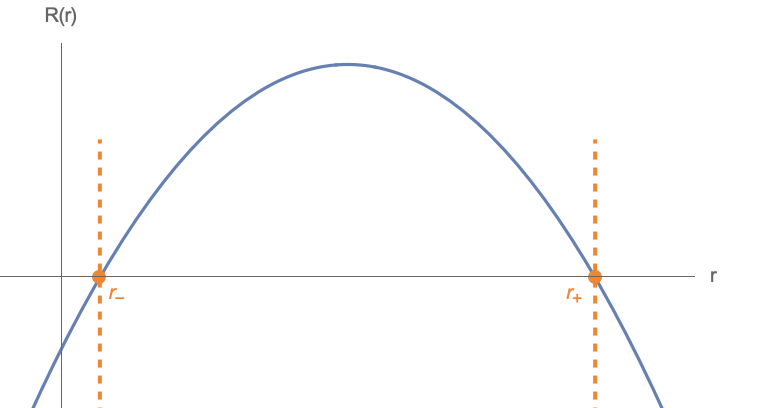} 
\caption{Plot of $R(r)$ in the case $E=0,\; L^2+Q\neq 0,\; L=0$\\ with $a=3,M=6, L=0, Q=4$.}
\label{R(r) E=0,  L=0}
\vspace*{-5mm}
\end{figure}

\hspace{1cm}

This polynomial $R(r)$ cannot produce a CNG since the hypersurfaces $\mathscr{H}=\{r=r_{\pm}\}$ are closed totally geodesic submaninfolds by Prop.  \ref{H is closed totally geod}  and a geodesic cannot have  turning points on such hypersurfaces because it would be tangent to them there.\\

\newpage

\subsection{Case $E\neq 0$}\label{section $E div 0$}

\vspace{0.8cm}
\subsubsection{Subcase $Q=0$} \label{case Q=0}

We have

\begin{align}
R(r)&=E^2r^4+(a^2E^2-L^2)r^2+2M(L-aE)^2r\geq 0,\\
\Theta(\theta)&=\cos^2\theta \bigg(a^2E^2-\frac{L^2}{\sin^2\theta}\bigg)\geq 0. \label{theta eq E non 0, Q=0}
\end{align}

\begin{prop}\label{prop r-bounded eq geodesics}
All the null geodesics with $E\neq 0,\; Q=0$ which have bounded radial behaviour lie in $Eq=\{\theta=\pi/2\}$.
\end{prop}

\begin{proof}
Suppose there exists an $r$-bounded null geodesic with $E\neq 0,Q=0$ for which $\theta(s)\neq \pi/2$ for some $s$.  Then from \eqref{theta eq E non 0, Q=0},  we get

\begin{align*}
a^2E^2\geq \frac{L^2}{\sin^2\theta}\geq L^2.
\end{align*}
\vspace{0.8cm}

Observe that 

\begin{align*}
R(0)&=0,\\
R'(r)&=4E^2r^3+2(a^2E^2-L^2)r+2M(L-aE)^2,\\
R''(r)&=12E^2r^2+2(a^2E^2-L^2)\geq 0.
\end{align*}
Therefore $R$ is convex and can only produce an $r$-bounded behaviour if the $r$-coordinate is constant $r(s)\equiv 0$,  when $r=0$ is a multiple root of $R(r)$,  \textit{i.e.} if $L=aE$.  The polynomial reduces then to $R(r)=E^2r^4$ (see Fig.\ref{R(r) E dif 0,  L=aE}).  However in this case we get $\Theta(\theta)=-\frac{\cos^4\theta}{\sin^2\theta}a^2E^2<0$ for some $s$.

 \begin{figure}[H]
\centering
\includegraphics[scale=0.6]{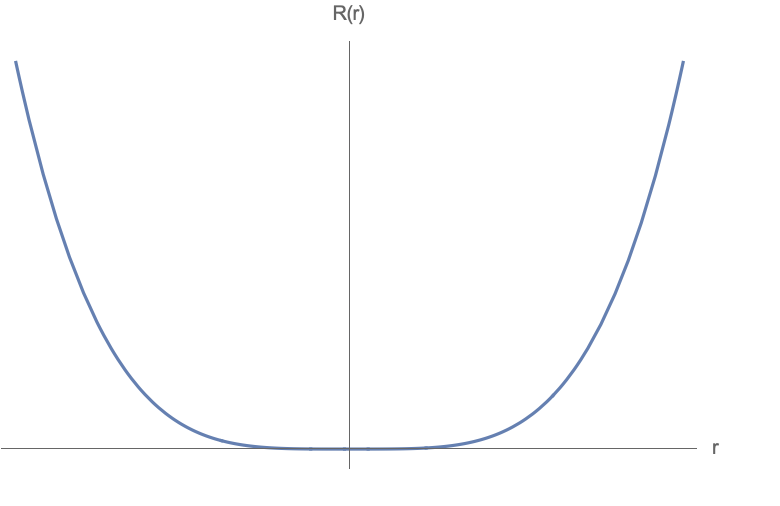} 
\caption{Plot of $R(r)$ in the case $E\neq 0, \; Q=0,\; L=aE\neq 0$ with $a=3,$\\$M=10, E=5, Q=0, L=15$.}
\label{R(r) E dif 0,  L=aE}
\vspace*{-5mm}
\end{figure}

\end{proof}

By Prop.  \ref{prop r-bounded eq geodesics},  we can suppose $\theta(s)=\pi/2$ for every $s$.  Since $R'(0)=2M(L-aE)^2$,   if $L\neq aE$  geodesics starting in $\{r<0\}$ are constrained in this region and cannot reach $\{r\geq 0\}$.  Hence there are no $r$-bounded null geodesics in the region $r<0$ by  Lemma $4.14.2$ \cite{KBH_book}.  Since the geodesics are constrained in $Eq$,  $r=0$ cannot be a $r$-turning point.  By  Lem.  \ref{lemma about closed curves confined},  there could be geodesics with bounded radial behaviour only in $\{0<r<r_{-}\}$.  However null geodesics constrained in the last region cannot be closed by Prop.  \ref{prop spacelike foliation}.
If instead $L=aE$,  then $R(r)=E^2r^4$.  Then the only possible bounded $r$-behaviour would be $r(s)\equiv 0$ which cannot happen as otherwise the geodesic would lie on the ring singularity.

\hspace{1cm}

\subsubsection{Subcase $Q>0$} \label{case Q>0} From Prop.  \ref{eq principal nulls},  we know that for null geodesics $K=Q+(L-aE)^2\geq 0$.  Hence if $Q>0$,  then $K>0.$ Let us again consider $R(r)=E^2r^4+\mathfrak{X}(E,L,Q)r^2+2MK(E,L,Q)r-a^2Q$.  Since $R(0)=-a^2Q<0$,  if a null geodesic has bounded $r$-behaviour  either it must be confined entirely in the negative region or entirely in the positive region,  since we must have $R(r)\geq 0$.  Bounded radial behaviour in the negative region is impossible in this subcase.  Indeed,  the signs of the coefficients of $R(-r)$ are $+\; \textrm{sign}(\mathfrak{X})\; -\; -$,  so for every sign of $\mathfrak{X}$,  there would be only one change of sign,  hence there can be only one single real negative root by the \textit{"Descartes' rule of signs"}.\footnote{ Wilkins \cite{Wilkins} was apparently the first to apply this rule to the polynomial $R$ from the radial equation of motion.}


By Lem.  \ref{lemma about closed curves confined},  we could only have a bounded (either in a compact $r$-interval or constant $r$) $r$-behaviour  in the positive region  $\{0<r<r_{-}\}$.  However null geodesics constrained in the last region cannot be closed by Prop.  \ref{prop spacelike foliation}.

\vspace{0.8cm}

\subsubsection{Subcase $Q<0$} \label{case Q<0}

This is the last remaining case and the most difficult one.  By Prop.  \ref{geod Q<0},  the only possible bounded behaviour is $r(s)=\textrm{const}<0$ (see  Fig.\ref{R(r) Q<0 r=const}).  Such geodesics  are known in the literature as \textit{spherical geodesics},  see e.g.  \cite{spherical_photon}.

\begin{figure}[H]
\centering
\includegraphics[scale=0.6]{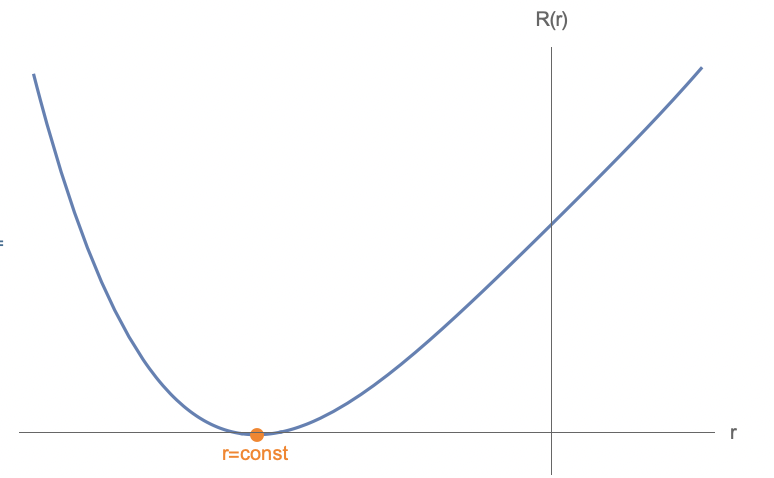} 
\caption{Plot of $R(r)$ in the case $E\neq 0$,  $Q<0$ with $a =\frac{1}{\sqrt{2}},  M = \frac{6}{\sqrt{23}},  E = 1, $\\ $L = 3 \sqrt{\frac{23}{2}},  Q = -104$.}
\label{R(r) Q<0 r=const}
\vspace*{-5mm}
\end{figure}

\hspace{1cm}

By Prop.  \ref{Q<0 condition},  null geodesics with negative Carter constant do not meet $Eq=\{\theta=\pi/2\}$,  hence $\cos^2\theta\neq 0$.  Then we may define $u:=\cos^2\theta\in (0,1]$.  Since we are in the case $E\neq 0$,  we can re-write the $\theta$-equation in \eqref{geodes diff equations} as

\begin{align}\label{theta di u}
\bigg( \frac{\rho^2(r,u)}{E}\bigg)^2\frac{(u')^2}{4u}=-a^2u^2+(a^2-\Phi^2-\mathcal{Q})u+\mathcal{Q}=:\tilde{\Theta}(u),
\end{align}
where $\Phi:=L/E$ and $\mathcal{Q}:=Q/E^2$.  Since we must have $\tilde{\Theta}(u)\geq 0$ somewhere in $(0,1]$,  $w:=a^2-\Phi^2-\mathcal{Q}>0$ beacuse $\mathcal{Q}<0$ and the coefficient of the second order term is negative.  Therefore $\tilde{\Theta}$ must have roots given by

\begin{align}\label{formula for u pm}
u_{\pm}=\frac{w\pm\sqrt{\textrm{dis}}}{2a^2}
\end{align}
where $\textrm{dis}:=w^2+4a^2\mathcal{Q}$,  so that 

\begin{align}
\tilde{\Theta}(u)=-a^2(u-u_{+})(u-u_{-}).
\end{align}
Hence we have

\[
0<u_{-}\leq u_{+}.
\]
We can write the necessary condition to have roots as

\begin{align*}
\textrm{dis}=[\mathcal{Q}+(|\Phi|-|a|)^2][\mathcal{Q}+(|\Phi|+|a|)^2]\geq 0.
\end{align*} 
Notice that $\mathcal{Q}+(|\Phi|+|a|)^2\geq \mathcal{Q}+(\Phi-a)^2=K/E^2\geq 0$ by Prop.  \ref{eq principal nulls}.  If we have $\mathcal{Q}+(|\Phi|+|a|)^2=0$,  then $\tilde{\Theta}(u)=-a^2(u-1)^2-\Phi^2-2|a\Phi|(1-u)\geq 0$ is satisfied only for $u=u_{+}=u_{-}=1$ and $\Phi=0$,  which falls into Prop.  \ref{u-=u+ Q<0 prop}.  If instead $\mathcal{Q}+(|\Phi|+|a|)^2>0$,  we must have  $\mathcal{Q}+(|\Phi|-|a|)^2\geq 0$.  Then by the AM-GM inequality we have
\begin{align*}
\sqrt{\textrm{dis}}\leq \frac{[\mathcal{Q}+(|\Phi|-|a|)^2]+[\mathcal{Q}+(|\Phi|+|a|)^2]}{2}=a^2+\Phi^2+\mathcal{Q},
\end{align*}
and so 

\begin{align*}
u_{+}\leq \frac{a^2-\Phi^2-\mathcal{Q} + a^2+\Phi^2+\mathcal{Q}}{2a^2}=1.
\end{align*}

Therefore we have

\begin{align*}
0<u_{-}\leq u_{+}\leq 1.
\end{align*}

\begin{prop}\label{u-=u+ Q<0 prop}
In the Kerr-star spacetime,  consider a null geodesic $\gamma$ with $\mathcal{Q}<0$ and $r=\textrm{const}$.  If $\textrm{dis}=0$,  then $\theta=\textrm{const}$ and the geodesic cannot be closed.
\end{prop}

\begin{proof}
Since $\textrm{dis}=0$,  we have $u_{-}=u_{+}$.  Then

\[
\tilde{\Theta}(u)=-a^2\big(u-u_{+}\big)^2\geq 0.
\]
Hence the last inequality is satisfied only if $u=u_{+}=\textrm{const}$,  therefore $\theta=\textrm{const}$.  Then the geodesic $\gamma$ cannot be closed.  Indeed,  there are two possibilities.  First,  if $\gamma$ is entirely contained in $A$,  then it cannot be closed by Prop.  \ref{no closed geodesics in axis}.  Second,  if $\gamma$ is not entirely contained in $A$,  by Prop.  \ref{differential equations of geodessics}  a geodesic of the form $s\mapsto (t(s),r_0,\theta_0,\phi(s))$ has $t'\equiv\textrm{const}$ and $\phi'\equiv\textrm{const}$.  It follows that $s\mapsto t(s)$ and $s\mapsto \phi(s)$ are affine functions.  If the geodesic is bounded in $K^*$,  then $t(s)$ must be constant.  Note that curves of the kind $\gamma(s)=(t_0,r_0,\theta_0,b_0s+b_1)$,  $b_0,b_1\in\mathbb{R}$,  are geodesics if and only if $b_0=0$ since the geodesic equation can be written in BL coordinates as $\Gamma^\alpha_{\phi\phi}(\gamma(s))b_0^2=0$ but the Christoffel symbol $\Gamma^{\theta}_{\phi\phi}$ cannot vanish at points where $\partial_\phi$ is null. Indeed,

\begin{align*}
\Gamma^{\theta}_{\phi\phi}&=-\frac{\sin\theta \cos\theta}{\rho^6(r,\theta)}\bigg[ \rho^4(r,\theta) \frac{\mathbf{g}(\partial_\phi,\partial_\phi)}{\sin^2\theta} +2M(r^2+a^2)a^2r\sin^2\theta \bigg]\\
&=-\frac{\sin\theta \cos\theta}{\rho^6(r,\theta)}\bigg[ 2M(r^2+a^2)a^2r\sin^2\theta \bigg]\neq 0,
\end{align*}
since $\theta\neq 0,\pi$ because we have already ruled out closed null geodesics in $A$,  $\theta\neq \pi/2$ by Prop.  \ref{Q<0 condition} and $r<0$.  Hence,  $\phi(s)$ is also constant and the geodesic degenerates to a point.
\end{proof}

\begin{oss}
Closed null curves exist in the Kerr-star spacetime: for instance,  they are given by the integral curves of the vector field $\partial_\phi$,  whenever this last happens to be null,  for some negative $r$.  Such curves cannot be geodesics by Prop.  \ref{u-=u+ Q<0 prop}.
\end{oss}

We may now assume $\textrm{dis}>0$.  Therefore we have the following chain of inequalities




\begin{align}
0<u_{-}<u_{+}\leq 1.
\end{align}


We hence define 

\begin{align}\label{values of the thetai}
\theta_1:=\arccos(\sqrt{u_{+}}),\theta_2:=\arccos(\sqrt{u_{-}}),\theta_3:=\arccos(-\sqrt{u_{-}}),\theta_4:=\arccos(-\sqrt{u_{+}})
\end{align}
so that
\begin{align}
0\leq \theta_1< \theta_2<\frac{\pi}{2}<\theta_3< \theta_4\leq \pi.
\end{align}

\begin{prop}\label{prop theta behaviours Q<0}
In the Kerr-star spacetime,  null geodesics with $\mathcal{Q}<0$,  $r=\textrm{const}$  and $\theta\neq\textrm{const}$ can have one of the following $\theta$-behaviours:

\begin{itemize}
\item if $0<u_{-}<u_{+}<1$,  then the $\theta$-coordinate oscillates periodically in one of the following intervals $0<\theta_1\leq \theta\leq \theta_2<\pi/2$ or $\pi/2<\theta_3\leq \theta\leq \theta_4<\pi$;
\item if $0<u_{-}<u_{+}=1$,  then $\Phi(=L/E)= 0$ and the $\theta$-coordinate oscillates periodically in one of the following intervals $0=\theta_1\leq \theta\leq \theta_2<\pi/2$ or $\pi/2<\theta_3\leq \theta\leq \theta_4=\pi$,
\end{itemize}
where the $\theta_i$,  $i=1,2,3,4$,  are given by \eqref{values of the thetai}.
\end{prop}

\begin{proof}
If $0<u_{-}<u_{+}<1$,  the graph of $\Theta(\theta)/E^2$ is shown in Fig.   \ref{fig:Theta if phi diff 0}.

\begin{figure}[H]
\centering
\includegraphics[scale=0.5]{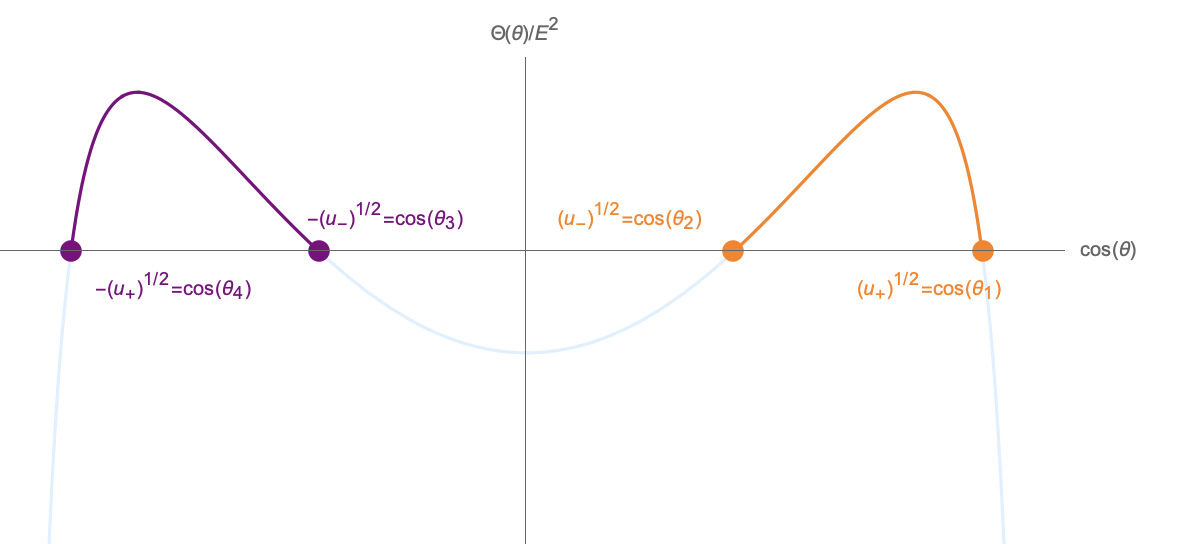} 
\caption{Plot of $\Theta(\theta)/E^2$ for $a=4,  \mathcal{Q}=-2,  |\Phi|=\sqrt{3}$.  The $\theta$-motion is allowed in the orange region,  above the equatorial hyperplane and in the purple region,  below it.\\}
\label{fig:Theta if phi diff 0}
\end{figure}
\noindent
The $\theta$-motion is allowed in the region where $\tilde{\Theta}(u)$ and hence $\Theta(\theta)/E^2$ are non-negative:

\begin{align*}
&0<\theta_1\leq \theta\leq \theta_2<\pi/2\hspace{0.5cm}\textrm{in the orange region,  above the}\;Eq,\\
&\pi/2<\theta_3\leq \theta\leq \theta_4<\pi\hspace{0.5cm}\textrm{in the purple region,  below the}\;Eq.
\end{align*}
This $\theta$-motion is shown in Fig.  \ref{fig:thetaSimulated_L diff 0} and the corresponding $\theta$-$\phi$-motion in Fig.  \ref{fig: theta phi motion 1}.

If $0<u_{-}<u_{+}=1$,  the graph of $\Theta(\theta)/E^2$ is shown in Fig.  \ref{fig:Theta if phi equal 0}.

\begin{figure}[H]
\centering
\includegraphics[scale=0.4]{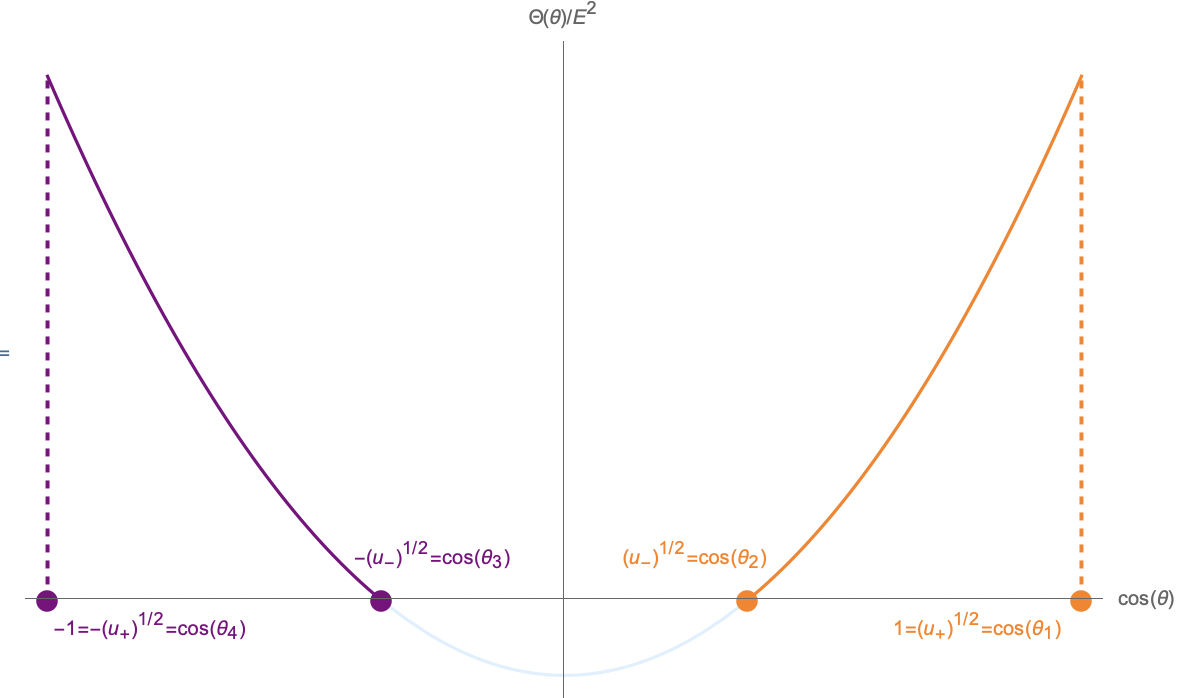} 
\caption{Plot of $\Theta(\theta)/E^2$ for $a=4,  \mathcal{Q}=-2,  \Phi=0$.  The $\theta$-motion is allowed in the orange region,  oscillating simmetrically around $\{\theta=0\}$ above the equatorial hyperplane and in the purple region,  oscillating simmetrically around $\{\theta=\pi\}$ below the equatorial hyperplane.\\}
\label{fig:Theta if phi equal 0}
\vspace*{1cm}
\end{figure}
\noindent
These geodesics intersect the axis $A=\{\theta=0,\pi\}$ and hence $\Phi=0$ because $\tilde{\partial}_\phi=0$ on $A$.  The $\theta$-motion is allowed where $\tilde{\Theta}(u)$ and hence $\Theta(\theta)/E^2$ are non-negative:

\begin{align*}
&0=\theta_1\leq \theta\leq \theta_2<\pi/2\hspace{0.5cm}\textrm{in the orange region,  above the}\;Eq,\\
&\pi/2<\theta_3\leq \theta\leq \theta_4=\pi\hspace{0.5cm}\textrm{in the purple region,  below the}\;Eq.
\end{align*}
\noindent
This $\theta$-motion is shown in Fig. \ref{fig:thetaSimulated_L=0} and the corresponding $\theta$-$\phi$-motion in Fig.  \ref{fig: theta phi motion 2}.

There is a difference between the motions of Fig.  \ref{fig:Theta if phi diff 0} and Fig.  \ref{fig:Theta if phi equal 0}.  In the first case,  the geodesics oscillates between $\theta_1$ and $\theta_2$ (or between $\theta_3$ and $\theta_4$),  corresponding to half $\theta$-oscillation.  In the second case instead,  the geodesics complete symmetric oscillations around the axis,  either above the equatorial hyperplane,  crossing $\{\theta=0\}$ or below the equatorial hyperplane crossing $\{\theta=\pi\}$.  However,  the motion between $\theta_1=0$ and $\theta_2$  (or between $\theta_3$ and $\theta_4=\pi$) still corresponds to half $\theta$-oscillation  (see Figures \ref{fig:thetaSimulated_L=0},  \ref{fig:Theta if phi equal 0} and  Cor.  $4.5.6$ in \cite{KBH_book}).  

In both oscillating cases,  since $r=\textrm{const}$,  \eqref{theta di u} implies that the coordinate $\theta(s)$ oscillates periodically in the corresponding $\theta$-interval (see Figures \ref{fig:thetaSimulated_L diff 0} and \ref{fig:thetaSimulated_L=0}).  
\end{proof}

Consider the first order equations of motion (with the rescaled constants of motion $\mathcal{Q}:=Q/E^2,  \Phi:=L/E$),  for a constant $r<0$:

\begin{align}
\label{equation of theta plot}\frac{\rho^2(r,\theta)}{E}\frac{d\theta}{ds}&=\pm\sqrt{\Theta(\theta)}=\pm\sqrt{\mathcal{Q}+a^2\cos^2\theta-\Phi^2\frac{\cos^2\theta}{\sin^2\theta}}\\
\label{equation of t plot}\frac{\rho^2(r,\theta)}{E}\frac{dt}{ds}&=\frac{r^2+a^2}{\Delta}(r^2+a^2-a\Phi)+a(\Phi-a\sin^2\theta).
\end{align}

Because of the $\theta$-differential equation,  we can restrict to an interval $\mathcal{U}\subset\theta^{-1}\big( (\theta_1,\theta_2)\big)$ on which $d\theta/ds$ is either everywhere positive or everywhere negative (depending on the initial condition).  Due to the symmetry in \eqref{equation of theta plot} and the fact that $r=\textrm{const}$,  $\theta(s)$ is periodic over twice the interval $\mathcal{U}$.  For instance,  set $\mathcal{U}=(0,T/2)$,   starting from $\theta(0)=\theta_1$,  hence $\theta'(s)=+\sqrt{\Theta(\theta)} >0$ for $s\in(0,T/2)$,  then $\theta'(s)=-\sqrt{\Theta(\theta)} <0$ for $s\in(T/2,T)$,  where $\theta'(T/2)=0$,  because $\theta(T/2)=\theta_2$ ($'\equiv d/ds$) and Prop.  \ref{initial conditions and zeroes},  which explains the change of sign of $\theta'(s)$ (using the fact that $\theta_1,\theta_2$ are multiplicity one zeroes of $\Theta(\theta)$).  Hence every $\Delta s=T/2$,  $\theta'(s)$ changes sign.

\begin{figure}[h]
\centering
\includegraphics[scale=0.5]{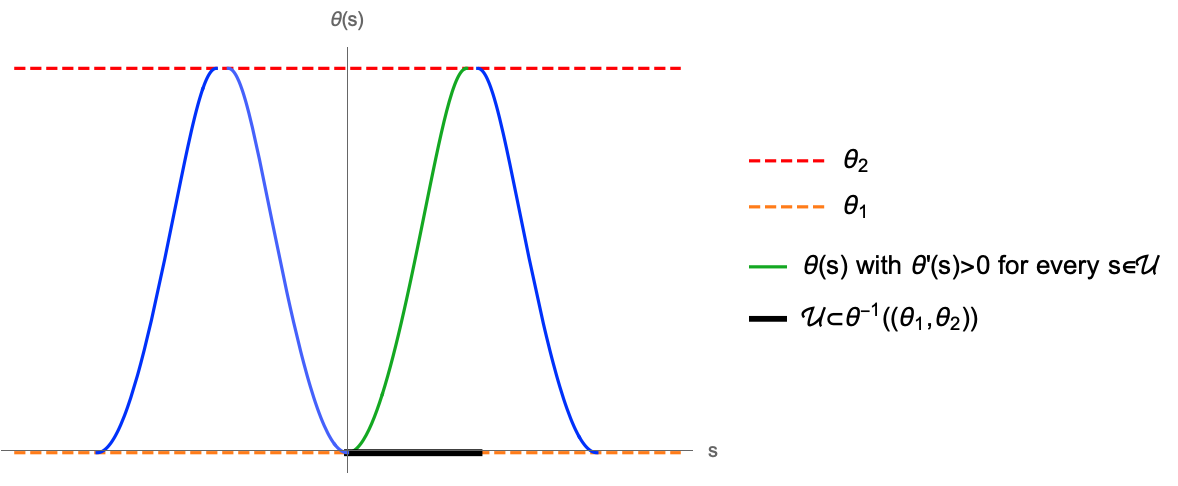} 
\caption{$\theta(s)$ obtained numerically from \eqref{equation of theta plot},  with $a=3, M=8, \\E=1,  \mathcal{Q}=-1.252,  \Phi=1.407$ and $r=-1$.  The open subset $\mathcal{U}$ (black in the figure) is an interval on which $\theta(s)$ is strictly monotonic.\\}
\label{fig:thetaSimulated_L diff 0}
\end{figure}

\begin{figure}[h]
\centering
\includegraphics[scale=0.5]{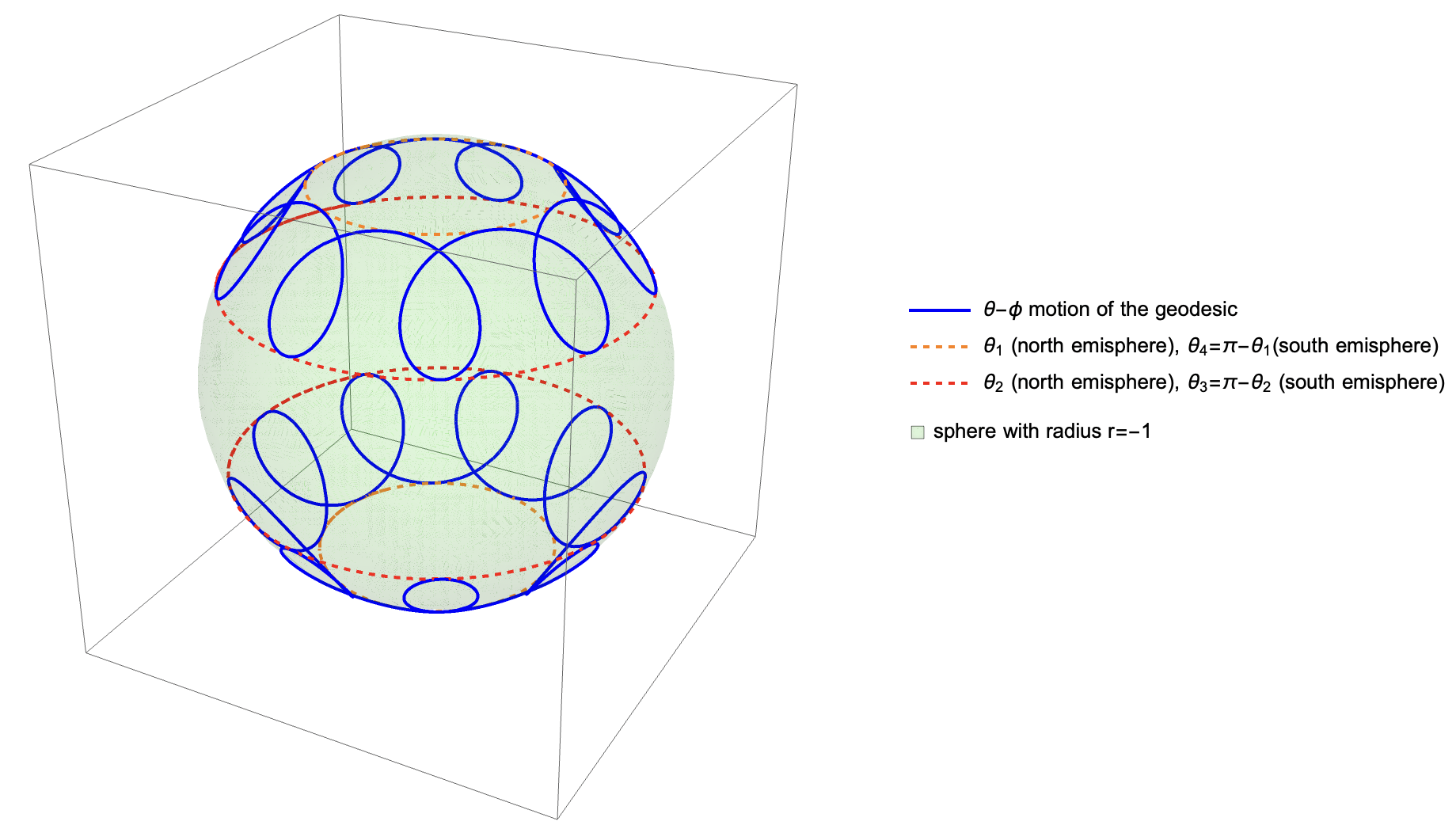} 
\caption{$\theta$-$\phi$ motion obtained numerically from the geodesic equations,  with $a=3,  M=8, E=1,  \mathcal{Q}=-1.252,  \Phi=1.407$ and $r=-1$ and initial conditions $\gamma(0)=(0,-1,\theta_1,0), \; \gamma'(0)\approx (0.529,0,0,0.376)$.}
\label{fig: theta phi motion 1}
\end{figure}


\begin{figure}[H]
\centering
\includegraphics[scale=0.5]{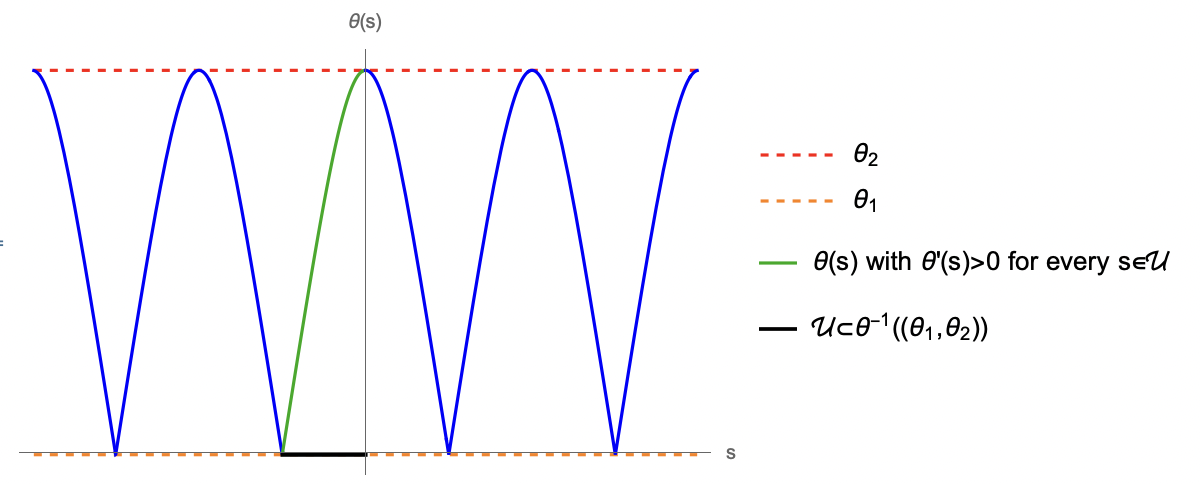} 
\caption{$\theta(s)$ obtained numerically from the geodesic equation with \\
$a=3, M=8,  r=r_{\textrm{crit}}<0$ where  $r_{\textrm{crit}}$ is the radius at which $\mathbf{g}(\partial_\phi,\partial_\phi)|_{r=r_{\textrm{crit}},\theta=\theta_2}=0$ and initial conditions $\gamma(0)=(0,r_{\textrm{crit}},\theta_2,0)$ and $\gamma'(0)=(0,0,0,-1)$.  Hence,  since $\gamma'(0)=-\partial_\phi$,  the geodesic has $\Phi=0$ ($L=0$).  The Carter constant is set to $\mathcal{Q}=-5.409$.  The open subset $\mathcal{U}$ (black in the figure) is an interval on which $\theta(s)$ is strictly monotonic.\\}
\label{fig:thetaSimulated_L=0}
\end{figure}

\begin{figure}[H]
\centering
\includegraphics[scale=0.5]{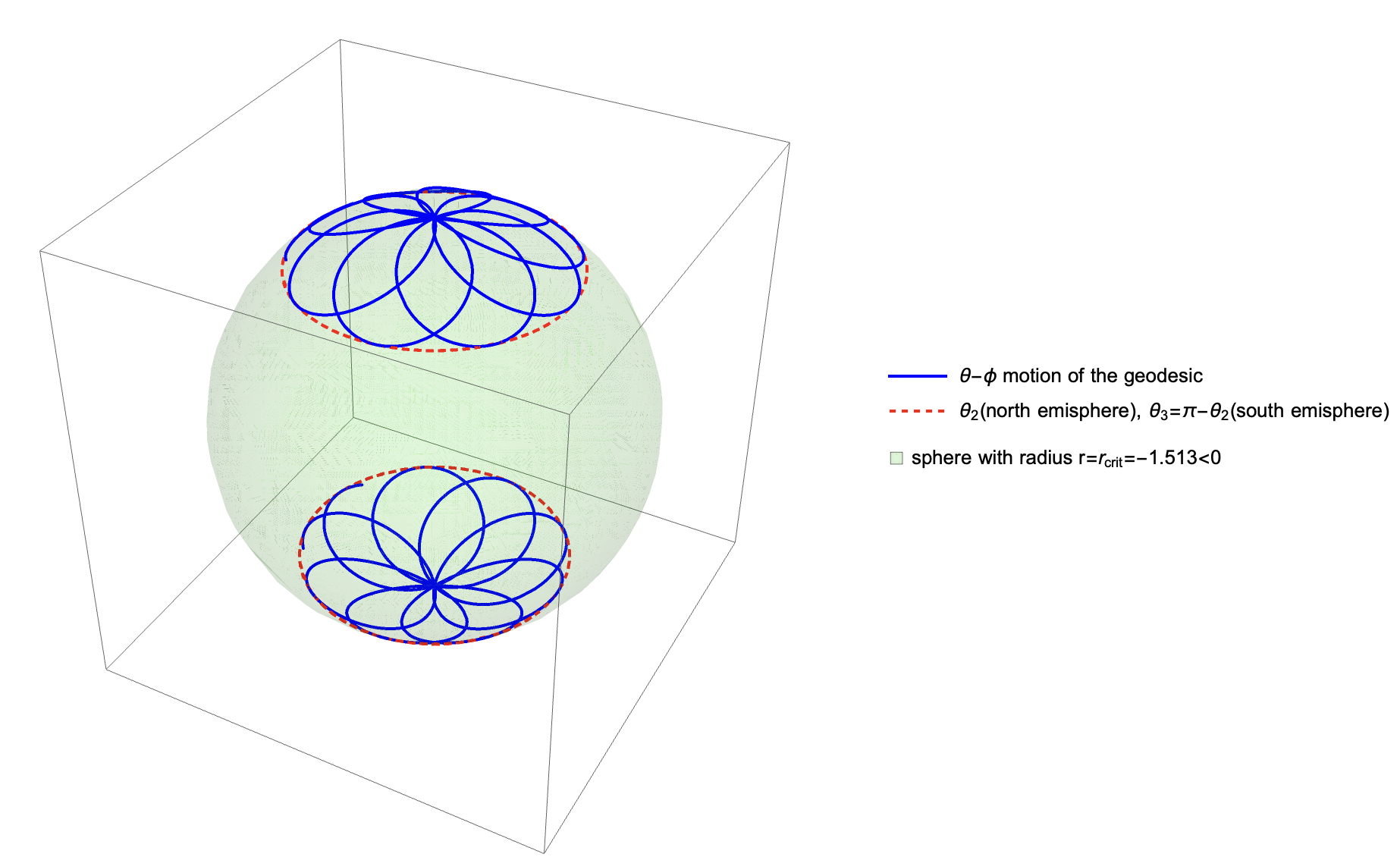} 
\caption{$\theta$-$\phi$ motion obtained numerically from the geodesic equation with \\ $a=3,M=8,  r=r_{\textrm{crit}}<0,  \mathcal{Q}=-5.409,  \Phi=0$ where  $r_{\textrm{crit}}$ is the radius at which $\mathbf{g}(\partial_\phi,\partial_\phi)|_{r=r_{\textrm{crit}},\theta=\theta_2}=0$ and initial conditions $\gamma(0)=(0,r_{\textrm{crit}},\theta_2,0)$ and $\gamma'(0)=(0,0,0,-1)$.  Notice that the geodesic crosses $\theta_1=0$ and $\theta_4=\pi$ with non-zero velocity.}
\label{fig: theta phi motion 2}
\end{figure}

At parameters where $\Theta(\theta)\neq 0$ we can combine \eqref{equation of t plot} and \eqref{equation of theta plot} to get

\begin{align}\label{t theta eq}
\frac{dt}{d\theta}=\frac{r^2\Delta+2Mr(r^2+a^2-a\Phi)}{\pm\Delta\sqrt{\Theta(\theta)}}+a^2\frac{\cos^2\theta}{\pm\sqrt{\Theta(\theta)}}=B(r)\frac{1}{\pm\sqrt{\Theta(\theta)}}+a^2\frac{\cos^2\theta}{\pm\sqrt{\Theta(\theta)}},
\end{align}
with $B(r):=\frac{r^2\Delta(r)+2Mr(r^2+a^2-a\Phi)}{\Delta(r)}.$ \\ 

\begin{prop}[See \cite{Chandrasekhar},  \cite{spherical_photon}]
In a Kerr spacetime,  a null geodesic with negative Carter constant $\mathcal{Q}$ and constant radial coordinate has the following pair of constants of motion
\begin{align}
\Phi=\Phi(r)= \frac{r^2(r-3M)+a^2(M+r)}{a(M-r)}\;,\hspace{1cm}\mathcal{Q}=\mathcal{Q}(r)=-r^3\frac{(r^3-6Mr^2+9M^2r-4a^2M)}{a^2(M-r)^2}.\label{class r=const}
\end{align}
\end{prop}

\begin{proof}
Since $E\neq 0$ by Prop.  \ref{Q<0 condition},  we can divide the $r$-equation by $E^2$ to get

\begin{align}
\bigg( \frac{\rho^2}{E}\bigg)^2(r')^2=r^4+(a^2-\Phi^2-\mathcal{Q})r^2+2M\bigg( (a-\Phi)^2+\mathcal{Q}\bigg)r-a^2\mathcal{Q}=:\mathcal{R}(r),
\end{align}
where $\Phi:=L/E$ and $\mathcal{Q}:=Q/E^2$.  A geodesic has constant radial behaviour if and only if  $\mathcal{R}(r)=0$ and $d\mathcal{R}(r)/dr=0$.  These two equations can be solved for $\mathcal{Q}$ and $\Phi$.  The two resulting pairs of constants of motion are \eqref{class r=const} and

\begin{align}
\Phi=\Phi(r)&=\frac{r^2+a^2}{a}\;,\hspace{3.68cm} \mathcal{Q}=\mathcal{Q}(r)=-\frac{r^4}{a^2},\label{class r=const wrong}
\end{align}
where $r$ is the constant radius of the geodesic.  Recall that $u_{-}+u_{+}=w/a^2>0$.  However \eqref{class r=const wrong} implies

\begin{align*}
w=-2r^2<0.
\end{align*}
Hence no null geodesic with constant radial coordinate and $\mathcal{Q}<0$ satisfies \eqref{class r=const wrong}.
\end{proof}

The first integral $\Phi=\Phi(r)$  is given by \eqref{class r=const}.  Hence $B(r)$ reduces to

\begin{align}\label{B(r) expression}
B(r)=\frac{r^2(3M+r)}{r-M}.
\end{align}

\begin{oss}
Notice that if $B(r)\geq 0$,  there would be nothing to prove.  Indeed,  in this case the $t$-coordinate would be non-decreasing,  hence non-periodic.  However $B(r)<0$ for a negative $r$ sufficiently close to zero.  (In fact,  this holds for all spherical geodesics with $Q<0$,  see Appendix \ref{appendix motivation}.) Moreover,  numerical simulation shows that the $t$-component is not necessarily monotonic for a $Q<0$ geodesic with constant $r<0$ close to zero,  see Fig.  \ref{figure non-monotonicity}.  Therefore we shall evaluate $\Delta t$ on a full $\theta$-oscillation.

\begin{figure}[H]
\centering
\includegraphics[scale=0.5]{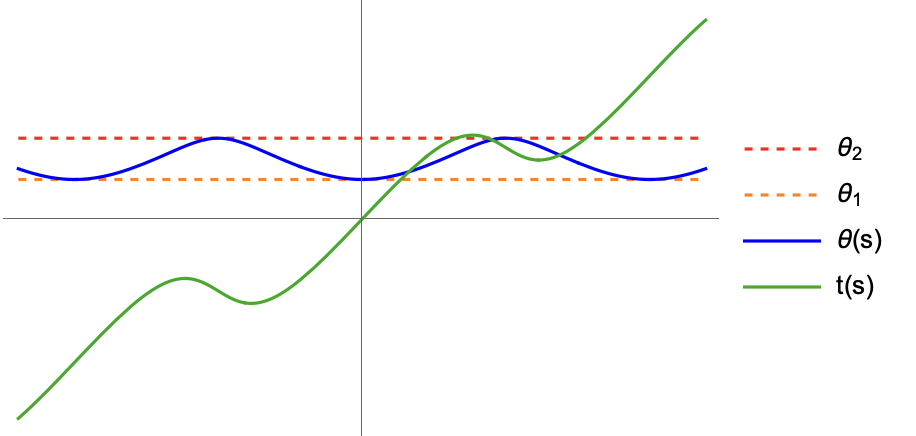} 
\caption{$\theta(s)$ obtained numerically from the geodesic equation with \\ $a=3,M=8,  r=\bar{r}=-1<0$  and initial conditions $\gamma(0)=(0,\bar{r},\theta_1,0)$ and $\gamma'(0)=(1,0,0,0.710)$.   Here $\mathcal{Q}=-1.252$ and $\Phi=1.407$.}
\label{figure non-monotonicity}
\vspace*{1cm}
\end{figure}
\end{oss}

\noindent
We are now finally ready to rule out constant radius geodesics in the subcase  $E\neq 0$,  $Q<0$.

By contradiction,  suppose there exists such a closed null geodesic $\gamma:I\rightarrow K^*$ with non-constant coordinate functions $s\mapsto t(s),\theta(s),\phi(s)$ and constant negative radial coordinate such that $B(r)<0$.  The differential equation \eqref{t theta eq} has the form

\begin{align*}
\frac{dt}{d\theta}=F(\theta),
\end{align*}
for some function $F$.  The variation of the $t$-coordinate on a full $\theta$-oscillation is given by 

\begin{align*}
\Delta t= 2\int_{\theta_1}^{\theta_2}F(\theta) d\theta.
\end{align*}

\begin{oss}\label{remark n oscillations}
Notice the factor $"2"$ in the last expression.  On a full $\theta$-oscillation,  we have 
\begin{align*}
\int_{\theta_1}^{\theta_2} F(\theta)d\theta +\int_{\theta_2}^{\theta_1} -F(\theta)d\theta=2\int_{\theta_1}^{\theta_2}F(\theta) d\theta.
\end{align*}
\end{oss}

Therefore the variation of the $t$-coordinate after $n$ $\theta$-oscillations is $n\Delta t$ because of the periodicity of the $\theta$-coordinate.  If the geodesic is closed,  $\Delta t=0$, otherwise the coordinate $t(s)$ cannot be periodic.  Hence it suffices to study what happens on a single $\theta$-oscillation.

\begin{oss}\label{same integrals}
A motion of the kind $\pi\geq \theta_4\geq\theta\geq\theta_3>\pi/2$ produces the same integrals since in this $\theta$-interval $\cos\theta<0$,  hence with the substitution $u=\cos^2\theta$  we have $d\theta=\frac{1}{2}\frac{du}{\sqrt{u}\sqrt{1-u}}$.  Therefore
\begin{align*}
\int_{\theta_1}^{\theta_2} \frac{d\theta}{\sqrt{\Theta(\theta)}}=\int_{\theta_3}^{\theta_4} \frac{d\theta}{\sqrt{\Theta(\theta)}},\hspace{1cm}  \int_{\theta_1}^{\theta_2} \frac{\cos^2\theta\;d\theta}{\sqrt{\Theta(\theta)}}=\int_{\theta_3}^{\theta_4} \frac{\cos^2\theta\;d\theta}{\sqrt{\Theta(\theta)}}.
\end{align*}
Hence $\Delta t$ is the same.
\end{oss}

So without any loss of generality,  we may consider a motion of the type $0\leq \theta_1\leq\theta\leq\theta_2<\pi/2$.  Then we can integrate \eqref{t theta eq} on a full oscillation to get

\begin{align}
\Delta t=2B(r)\int_{\theta_1}^{\theta_2}\frac{d\theta}{\sqrt{\Theta(\theta)}}+2a^2 \int_{\theta_1}^{\theta_2}\frac{cos^2\theta\; d\theta}{\sqrt{\Theta(\theta)}}.
\end{align}

We now have to compute the following integrals

\begin{align*}
I_1:=&\int_{\theta_1}^{\theta_2} \frac{d\theta}{\sqrt{\Theta(\theta)}},\\
I_2:=&\int_{\theta_1}^{\theta_2} \frac{\cos^2\theta\;d\theta}{\sqrt{\Theta(\theta)}}.
\end{align*}

Let us start from the first integral:
\begin{align}
I_1=-\frac{1}{2}\int_{u_{+}}^{u_{-}}\frac{du}{\sqrt{u}\sqrt{\tilde{\Theta}(u)}},
\end{align}
where we have used the substitution $u:=\cos^2\theta$,  hence $d\theta=-\frac{1}{2}\frac{du}{\sqrt{u}\sqrt{1-u}}$ since\\ $\sin\theta\geq 0$ and $\cos\theta>0$ if $\theta_1\leq \theta\leq\theta_2$.  Now we can use \eqref{theta di u} and the substitution\\ $u=:u_{-}+(u_{+}-u_{-})y^2$ adopted in \cite{Kapec_2019} to get

\begin{align*}
\nonumber   I_1=&\frac{1}{2}\int_{u_{-}}^{u_{+}}\frac{du}{\sqrt{u}\sqrt{a^2(u_{+}-u)(u-u_{-})}}\\ 
\nonumber =&\frac{1}{2|a|}\int_0^1 \frac{2(u_{+}-u_{-})ydy}{\sqrt{u_{-}+(u_{+}-u_{-})y^2}\sqrt{\big(u_{+}-u_{-}-(u_{+}-u_{-})y^2\big)(u_{+}-u_{-})y^2}}\\ 
\nonumber =& \frac{1}{|a|}\int_0^1 \frac{dy}{\sqrt{u_{-}+(u_{+}-u_{-})y^2}\sqrt{1-y^2}}\\ \nonumber
 \nonumber=& \frac{1}{|a|\sqrt{u_{-}}} \int_0^1 \frac{dy}{\sqrt{1-y^2}\sqrt{1-\big( 1-\frac{u_{+}}{u_{-}} \big) y^2}}.
\end{align*}
With the same substitutions,  we also get

\begin{align*}
\nonumber I_2=&-\frac{1}{2}\int_{u_{+}}^{u_{-}}\frac{udu}{\sqrt{u}\sqrt{\tilde{\Theta}(u)}}\\
\nonumber =&\frac{1}{2} \int_{u_{-}}^{u_{+}}\frac{udu}{\sqrt{u}\sqrt{a^2(u_{+}-u)(u-u_{-})}}\\
\nonumber =& \frac{1}{|a|}\int_0^1 \frac{\sqrt{u_{-}+(u_{+}-u_{-})y^2}}{\sqrt{1-y^2}}dy\\
\nonumber =& \frac{\sqrt{u_{-}}}{|a|}\int_0^1 \frac{\sqrt{1-\big(1-\frac{u_{+}}{u_{-}} \big)y^2}}{\sqrt{1-y^2}}dy.\\
\end{align*}

Then with the definition of the elliptic integrals in Appendix \ref{elliptic appendix}  we have

\begin{align}
I_1=&\frac{1}{|a|\sqrt{u_{-}}} \mathcal{K}\bigg(1-\frac{u_{+}}{u_{-}} \bigg),\\
I_2=&\frac{\sqrt{u_{-}}}{|a|} \mathcal{E} \bigg( 1-\frac{u_{+}}{u_{-}} \bigg).
\end{align}

Hence,  we get

\begin{align}
\Delta t=\frac{2B(r)}{|a|\sqrt{u_{-}}} \mathcal{K}\bigg(1-\frac{u_{+}}{u_{-}} \bigg) +  2|a|\sqrt{u_{-}} \mathcal{E} \bigg( 1-\frac{u_{+}}{u_{-}} \bigg).
\end{align}

Note that,  since $u_{+}>u_{-}>0$,  we have $1-\frac{u_{+}}{u_{-}}<0$,  and hence $\mathcal{E}(1-u_{+}/u_{-})>\mathcal{K}(1-u_{+}/u_{-})>0$ (see Appendix \ref{elliptic appendix}).
However,  the prefactor of $\mathcal{E}$ does not dominate the opposite of the prefactor of $\mathcal{K}$ for every negative $r$,  as one may check substituting $\Phi(r)$ and $\mathcal{Q}(r)$ from \eqref{class r=const} into $u_{-}$ given by \eqref{formula for u pm}.  

From now on set $x:=1-u_{+}/u_{-}$.  The elliptic integral $\mathcal{K}$ can be written as a hypergeometric function (see \ref{elliptic as hypergeom}): 

\begin{align*}
\mathcal{K}(x)=\frac{\pi}{2}F\bigg(\frac{1}{2},\frac{1}{2};1;x\bigg).
\end{align*}
Using the Pfaff transformation (see \ref{hypergeometric trick})

\begin{align}\label{hypergeo trick}
F\bigg( \alpha,\beta;\gamma;x\bigg)=(1-x)^{-\alpha}F\bigg( \alpha,\gamma-\beta;\gamma;\frac{x}{x-1} \bigg),
\end{align}
we can decrease the modulus of the prefactor in front of the elliptic integral $\mathcal{K}$:

\begin{align}\label{transformation k}
\mathcal{K}(x)=\frac{\sqrt{u_{-}}}{\sqrt{u_{+}}}\mathcal{K}\bigg(\frac{x}{x-1}\bigg).
\end{align}
Hence we get

\begin{align}\label{last espression for delta t}
\Delta t=2|a|\sqrt{u_{-}}\mathcal{E}(x)+\frac{2B(r)}{|a|\sqrt{u_{+}}}\mathcal{K}\bigg(\frac{x}{x-1}\bigg).
\end{align}
Now we compare the elliptic integrals,  after the Pfaff transformation.  Since $x<0$,  we have

\begin{align}\label{inequality elliptic integrals}
\mathcal{E}(x)>\mathcal{K}\bigg(\frac{x}{x-1}\bigg)>0,
\end{align}
by Rmk.   \ref{remark integral estimate}.  Next we claim that the prefactors of the elliptic integrals in  \eqref{last espression for delta t} satisfy

\begin{align}\label{ineq prefactors}
2|a|\sqrt{u_{-}}>-\frac{2B(r)}{|a|\sqrt{u_{+}}}.
\end{align}
Indeed,  both sides of the inequality are positive,  so we can square them and use that $u_{+}u_{-}=-\mathcal{Q}(r)/a^2$ by \eqref{theta di u} to get an equivalent inequality

\begin{align*}
-\mathcal{Q}(r)a^2>B^2(r),
\end{align*}
where $\mathcal{Q}(r)$ is given by \eqref{class r=const} and $B(r)$ by \eqref{B(r) expression} ,  or equivalently 

\begin{align*}
(-12Mr^2-4a^2M)r>0.
\end{align*}
This last inequality is clearly satisfied in $r<0$.  Combining  \eqref{last espression for delta t},  \eqref{inequality elliptic integrals} and \eqref{ineq prefactors},  we conclude that $\Delta t>0$ for all $r<0$,  which shows that the spherical geodesics cannot be closed.  

In Fig. \ref{plot Delta t(r)} we see the plot of $\Delta t$ given by \eqref{last espression for delta t} as function of the fixed radius $r$ after the substitutions of $\Phi(r)$ and $\mathcal{Q}(r)$ from \eqref{class r=const} into $u_{\pm}$ given by \eqref{formula for u pm}.

\begin{figure}[H]
\centering
\includegraphics[scale=0.6]{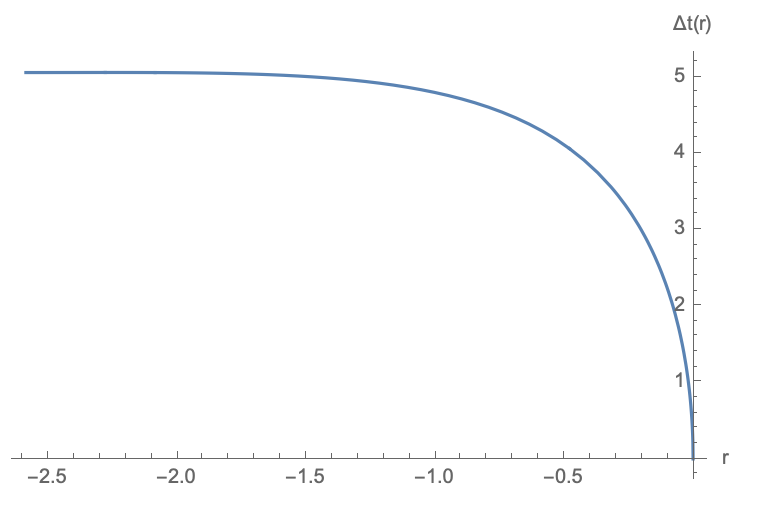} 
\caption{Plot of $\Delta t=\Delta t(r)$ for $r<0$ close to zero with $a=5,\,M=7$. }
\label{plot Delta t(r)}
\vspace{5mm}
\end{figure}

We have ruled out all the possibilities on Fig. \ref{figure steps of proof},  therefore there are no closed null geodesics in the Kerr-star spacetime.

\section{(Un)boundedness of null geodesics in the  Kerr-star spacetime}\label{section unbounded}

\subsection{Geodesics approaching a horizon}

\begin{defn}
A maximal geodesic $\gamma:(a,b)\rightarrow K^*$,  not entirely contained in the horizon,  approaches a horizon if 
\[
\lim_{s\to b^-}r\big(\gamma(s)\big)=r_{\pm}\hspace{1cm}\textrm{or}\hspace{1cm}\lim_{s\to a^+}r\big(\gamma(s)\big)=r_{\pm}.
\]
\end{defn}

\begin{prop}\label{prop unbounded geodesics approaching horizons}
Let $\gamma:[a,b)\rightarrow K^*$ be a null geodesic approaching a horizon $\mathscr{H}=\{r=r_{\pm}\}$ as $s\rightarrow b^-$.  Then $\gamma$ is unbounded.
\end{prop}

\begin{oss}\label{oss one side of horizon}
A geodesic can cross a horizon at most once.  Therefore we may assume that the image $\gamma([a,b))$ is on one side of a horizon,  increasing the parameter $a$ if necessary.
\end{oss}

\begin{proof}

Let $\gamma$ be a bounded null geodesic and let $\mathbf{h}$ be a fixed (arbitrary) auxiliary Riemannian metric on $K^*$.  Choose a sequence $s_n\rightarrow b$ such that $r\big( \gamma(s_n)\big)\rightarrow r_{\pm}$ and $\gamma(s_n)\rightarrow x_0$.  Then $x_0\in\mathscr{H}$.  Possibly passing to a new subsequence $s_m\rightarrow b$,  we may assume
\[
c_m \gamma'(s_m)\rightarrow v\in T_{x_0}K^*,
\]
with $c_m$ are $\mathbf{h}$-normalizing constants.  Note that $c_m \gamma'(s_m)=\beta'_m(0)$ where $\beta_m(s):=\gamma(s_m+c_m s)$ is an affine reparametrization of $\gamma$ with $\beta_m(0)=\gamma(s_m)$.  Now we distinguish two cases.\\
Case 1: $v\notin T_{x_0}\mathscr{H}$.  Then the geodesic $\beta$ with $\beta(0)=x_0$ and $\beta'(0)=v$ crosses the horizon.  By the continuous dependence of solution of an ODE on initial conditions and $\beta_m(0)\rightarrow x_0,  \beta_m'(0)\rightarrow v$,  $\beta_m$ must also cross the horizon for $m\rightarrow +\infty$.  However $Im\;\beta_m=Im\;\gamma$,  so $\gamma$ crosses the horizon,  which contradicts Rmk.  \ref{oss one side of horizon}.\\
Case 2: $v\in T_{x_0}\mathscr{H}$.  The geodesic $\beta$ with $\beta(0)=x_0$ and $\beta'(0)=v$ lies in the horizon,  hence it is a restphoton by Prop. \ref{H is closed totally geod} and is unbounded by Prop.  \ref{restphotons are unbounded}.  Fix now a compact interval $[a',b']$ so that $\beta([a',b'])$ is not in a given compact set.  Then by the  continuous dependence of solution of an ODE on initial conditions,  $\beta_m([a',b'])$ is not contained in the given compact set as $m\rightarrow +\infty$,  so the same holds also for $Im\;\gamma$,  which contradicts the boundedness of $\gamma$.
\end{proof}

\begin{cor}\label{incomplete are unbounded}
Incomplete null geodesics in $K^*$ are not bounded.
\end{cor}

\begin{proof}
Let $\gamma$ be an incomplete null geodesic.  By Prop. \ref{prop unbounded geodesics approaching horizons},  we can assume that $\gamma$ is not approaching a horizon.  Then $\gamma$ either crashes into the ring singularity or is contained in the axis $A$ or in a horizon $\mathscr{H}$ by Prop.  $4.3.9$ of \cite{KBH_book}.  If $\gamma$ crashes into the singularity,  by definition it cannot be contained in a compact set of $K^*$,  hence cannot be bounded.  If $\gamma\in A$,  the incompleteness of $\gamma$ in $K^*$ can only arise when the geodesic is approaching a horizon,  hence it cannot be bounded by Prop. \ref{prop unbounded geodesics approaching horizons},  see Lemma $4.11.2$ and preceding discussion in \cite{KBH_book}.  If instead $\gamma\in\mathscr{H}$,  it is a restphoton by Prop. \ref{H is closed totally geod} and by Prop.  \ref{restphotons are unbounded} it is not bounded.  
\end{proof}

\subsection{Radial behaviour of complete null geodesics}

\begin{prop}\label{prop.  r1 and r2}
Let $\gamma:\mathbb{R}\rightarrow K^*$ be a complete null geodesic and set $\inf_{s\in\mathbb{R}} r(\gamma(s))=r_1$ and $\sup_{s\in\mathbb{R}} r(\gamma(s))=r_2$.  Then $r_1$ and $r_2$ are either zeroes of the associated polynomial $R(r)$ or infinite.
\end{prop}

\begin{proof}
Let $\gamma:\mathbb{R}\to K^*$ be a complete null geodesic.  Assume $r_1:=\inf_{s\in\mathbb{R}} r(\gamma(s))$ is finite.  We want to prove that $r_1$ is a zero of the polynomial $R(r)$ associated to $\gamma$. \\

If $r(\gamma(s))=\textrm{const}$ for every $s$,  then this $\textrm{const}=r_1$ and the result follows from Prop.  \ref{differential equations of geodessics}.\\

If $r(\gamma(s))\neq\textrm{const}$,  there is an $s_0\in\mathbb{R}$ such that $r'(s_0):=r'(\gamma(s_0))\neq 0$.  Changing the orientation on $\gamma$ and shifting the parameter,  if needed,  we may assume that $s_0=0$ and $r'(s_0)<0$.  If $r'(s)$ changes sign at some $s_1>0$,  then $r(s_1)$ is a simple zero of $R(r)$ and $r_1=r(s_1)$ by Prop.  \ref{initial conditions and zeroes}.  If instead $r'(s)<0$ for all $s>0$,  then $r(s)$ is decreasing to $r_1^*:=\inf_{s\geq 0}r(\gamma(s))\geq r_1$ as $s\to +\infty$.  Since the integral of $r'(s)$ converges over $[0,+\infty)$,  there is a sequence $s_k\to +\infty$ such that $r'(s_k)\to 0$.  From the $r$-equation of Prop.  \ref{differential equations of geodessics},  it follows that $R(r(s_k))\to 0$ and hence $R(r_1^*)=0$.  However,  the $r$-coordinate of a geodesic cannot cross a zero of $R(r)$ by Prop.  \ref{initial conditions and zeroes},  so we must have $r_1=r_1^*$ and $R(r_1)=0$.\\

A completely analogous argument shows that $r_2$ is either infinite or a zero of $R(r)$.

\end{proof}


\begin{cor}\label{bounded could exist only in compact}
The polynomial $R(r)$ of a complete bounded null geodesic is non-negative on the compact interval $[r_1,r_2]$ and $R(r_1)=R(r_2)=0$.
\end{cor}
Now we state a proposition which is the analog of Prop.  \ref{prop spacelike foliation} for complete null geodesic rays.

\begin{prop}\label{no bounded geodesics with foliation}
If a complete null geodesic $\gamma:[0,+\infty)\rightarrow K^*$ has at least one non-vanishing constant of motion among $E,  L,  K(E,L,Q)$ and its radial behaviour $r(s)$ is compactly contained in $\{0<r<r_{-}\}$ or in $\{r>r_{+}\}$,  then it cannot be bounded.
\end{prop}

\begin{proof}
By contradiction,  let $\gamma: [0,+\infty)\rightarrow \{0<r<r_{-}\}\subset K^*$ (or $\gamma: [0,+\infty)\rightarrow \{r>r_{+}\}\subset K^*$) be a bounded complete null geodesic with at least one non-vanishing constant of motion.  The boundedness of the $t$-coordinate implies that there exists a sequence $s_n\rightarrow +\infty$ such that $dt\big(\gamma'(s_n)\big)\rightarrow 0$.  Now choose a subsequence $s_m\rightarrow +\infty$ such that $\gamma(s_m)\rightarrow p_{\infty}\in K^*$ as $m\rightarrow +\infty$ which always exists because the $Im\,\gamma$ is contained in a compact set.  Let $\mathbf{h}$ be some (arbitrary) auxiliary Riemannian metric on $K^*$.  Possibly passing to a new subsequence,  we may assume that 

\begin{align*}
\frac{\gamma'(s_m)}{|\gamma'(s_m)|_{\mathbf{h}}}\rightarrow v\in T_{p_\infty} K^*,
\end{align*}
where $v\neq 0$ and null because it is limit of $\mathbf{h}$-norm $1$ null vectors.  Moreover $\gamma'(s_m)$ cannot converge to zero because either $E=-\mathbf{g}(\partial_t,\gamma'(s))\neq 0$ or $L=\mathbf{g}(\partial_\phi,\gamma'(s))\neq 0$ or $K(E,L,Q)=2\rho^2(r,\theta)\mathbf{g}(l,\gamma'(s))\mathbf{g}(n,\gamma'(s))\neq 0$ (see \S \ref{study of geodesic equations}) is constant on $\gamma$. Therefore there exists an $\epsilon >0$ and a further subsequence $s_{k}\rightarrow +\infty$ such that $|\gamma'(s_k)|_{\mathbf{h}}\geq \epsilon$ for all $k$.  For this subsequence,  we have

\begin{align*}
|dt(v)|=\lim_{k\rightarrow +\infty} \frac{|dt\big( \gamma'(s_k)\big)|}{|\gamma'(s_k)\big)|_{\mathbf{h}}}\leq \frac{1}{\epsilon}  \lim_{k\rightarrow +\infty} |dt\big( \gamma'(s_k)\big)|=0.
\end{align*}
Then the null vector $v\in T_{p_\infty}\{t=\textrm{const}\}$,  which is a contradiction because $\{t=\textrm{const}\}$ is spacelike in both  regions  $\{0<r<r_{-}\}$ or $\{r>r_{+}\}$.\\
\end{proof}

\begin{prop}[\cite{KBH_book},  Lemma $4.7.3$]\label{prop roots R}
The polynomial $R(r)$ of every null geodesic (different from restphotons) in $K^*$ is positive on $(r_{-},r_{+})$ and has at most simple zeroes at $r_{\pm}$.
\end{prop}
\begin{proof}
From Prop.  \ref{differential equations of geodessics},  the polynomial $R(r)$ of null geodesics is
\[
R(r)=-K(E,L,Q)\Delta(r)+\mathbb{P}^2(r).
\]
Observe that $K(E,L,Q)\geq 0$ by Prop. \ref{eq principal nulls} and $\Delta(r)<0$ if $r\in(r_{-},r_{+})$.  Hence $R(r)$ is sum of non-negative terms if $r\in(r_{-},r_{+})$.  Now we distinguish two cases.\\

Case 1: $K(E,L,Q)=0$.  Then $Q=-(L-aE)^2\leq 0$.  Hence there are two subcases.  If $Q=0$,  then $\mathfrak{X}(E,L,Q)=0$.  From Prop.  \ref{differential equations of geodessics} we get $R(r)=E^2r^4>0$ for $r\neq 0$ if the geodesic is not a restphoton.  If $Q<0$,  $R(r)>0$ for all $r>0$ by Prop.  \ref{geod Q<0}.

Case 2: $K(E,L,Q)>0$.  Since $\Delta(r)<0$ for $r\in(r_{-},r_{+})$,  $R(r)>0$ in $r\in(r_{-},r_{+})$.\\
This proves the first claim. \\

Finally,  $R(r_{\pm})=\mathbb{P}^2(r_{\pm})=0$ if and only if $L=L_{\pm}=\frac{2Mr_{\pm}E}{a}$ is a root of $\mathbb{P}(r_{\pm})=(r_{\pm}^2+a^2)E-aL=0$.  Then 

\[
\frac{\partial R}{\partial r}\bigg|_{(r_{\pm},L_{\pm})}=-2(r_{\pm}-M)(r^2_{\pm}+K(E,L,Q))\neq 0.
\]

\end{proof}

\begin{prop}\label{prop bounded have r2<r+}
Let $\gamma: \mathbb{R}\rightarrow K^*$ be a complete bounded null geodesic and $r_2:=\sup_{s\in\mathbb{R}}r(\gamma(s))$. Then $r_2<r_{-}$. 
\end{prop}

\begin{proof}
Case (1): $r_2>r_{+}$.   If $r_1>r_{+}$,  then $r(\gamma(s))$ is compactly contained in $\{r>r_{+}\}$ for every $s$,  hence $\gamma$ cannot be bounded by Prop.  \ref{no bounded geodesics with foliation}.  By Prop.  \ref{prop roots R},  we hence must have $r_1\leq r_{-}$.
Since $\gamma$ can only cross $\mathscr{H}_{+}$ once,  say at $s_0$,  then  $r_2$ cannot be a turning point,  otherwise $\gamma$ would reverse its $r$-motion by Prop.  \ref{initial conditions and zeroes} at $r_2$ and cross $\mathscr{H}_{+}$ again.  Therefore $\gamma((-\infty,s_0-\epsilon))$,  for some $\epsilon>0$,  is compactly contained in the region $\{r>r_{+}\}$ and $\gamma$ cannot be bounded by Prop.  \ref{no bounded geodesics with foliation}. \\

Case (2): $r_2=r_{+}$ or $r_2=r_{-}$.  Then $\gamma$ is either contained in a horizon or approaching it ($r_{\pm}$ cannot be a turning point because the horizons are closed totally geodesic submanifolds.) Hence,  $\gamma$ is unbounded by Prop.  \ref{restphotons are unbounded} and Prop.  \ref{prop unbounded geodesics approaching horizons}.\\

Case (3): $r_2\in(r_{-},r_{+})=:I$.  This case is impossible since $R(r)>0$ on $I$ by Prop.  \ref{prop roots R}.  
\end{proof}

\section{Proof of Theorem \ref{second theorem}} \label{proof second theorem}
We assume that $\gamma$ is a bounded null geodesic and argue by contradiction.  By Cor.  \ref{incomplete are unbounded},  we may assume that $\gamma$ is complete and define $r_1$ and $r_2$ as in Prop.  \ref{prop.  r1 and r2}.  By Cor.  \ref{bounded could exist only in compact},  the polynomial $R(r)$ of $\gamma$ is non-negative on $[r_1,r_2]$ with $R(r_1)=R(r_2)=0$.\\
We split the argument in the cases as in the proof of Thm.  \ref{main theorem} in \S \ref{section: main theorem}.

\begin{figure}[h]
\centering
\begin{tikzpicture}[grow=right]
\tikzset{level distance=100pt,sibling distance=90pt}
 \tikzset{frontier/.style={distance from root=200pt}}

\Tree 
               [.$E=0$ [.$K(0,L,Q)>0$ ] 
                       [.$K(0,L,Q)=0$ ] 
                       ]   
                         
\end{tikzpicture}
\hspace{1cm}
\begin{tikzpicture}[grow=right]
\tikzset{level distance=80pt,sibling distance=40pt}
 \tikzset{frontier/.style={distance from root=300pt}}

\Tree 
               [.$E\neq 0$  [.$Q<0$ ] 
               		              [.$Q>0$ ] 
                                    [.$Q=0$ ] 
                     ]

\end{tikzpicture}
\vspace*{5mm}
\caption{All the geodesic types which have to be studied.}
\label{figure steps of proof}
\vspace*{-5mm}
\end{figure}
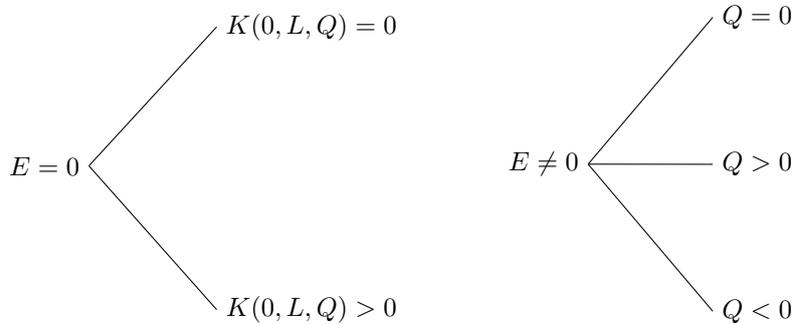

\vspace{1cm}

\subsection{Null geodesics in the horizons and axis}

A null geodesic in a horizon is a restphoton and is unbounded by Prop.  \ref{restphotons are unbounded}.  Note that by Prop.  \ref{E=L=K=0 prop} those are the only geodesics for which all the constants of motion vanish.  If $\gamma\in A\setminus\mathscr{H}$,  then $K=L=0,  E\neq 0$ by Prop.  \ref{E=L=K=0 prop},  hence $R(r)=E^2(r^2+a^2)^2\neq 0$ for every $r$.  Then $\gamma$ is unbounded by Cor.  \ref{bounded could exist only in compact}.

\subsection{Case $E=0$}
\hfill \break
\subsubsection{Subcase $K(0,L,Q)=0$,  cf.  \ref{subcase E=0,K=0}}
\hfill \break

The geodesic is not a restphoton,  so $L\neq 0$.  The argument in \ref{subcase E=0,K=0} then shows,  using Prop. \ref{prop bounded have r2<r+} instead of Lemma \ref{lemma about closed curves confined},  that $t'(s)=\textrm{const}\neq 0$,  hence $\gamma$ is unbounded


\subsubsection{Subcase $K(0,L,Q)>0$,  cf.  \ref{subcase E=0,K>0}}
\hfill \break

In this case,  $R(r)>0$ on $(r_{-},r_{+})$,  see Fig.  \ref{R(r) E=0,  L^2+Q dif 0,  Q>0},  \ref{R(r) E=0,  L^2+Q dif 0,  Q=0},  \ref{R(r) E=0,  L=0},  so $r_{2}\geq r_{+}$ and $\gamma$ cannot be bounded by Prop.  \ref{prop bounded have r2<r+}.



\subsection{Case $E\neq 0$}
\hfill \break
\subsubsection{Subcase $Q=0$,  cf.  \ref{case Q=0}}
\hfill \break

Such geodesics lie in $Eq=\{\theta=\pi/2\}$ by Prop. \ref{prop r-bounded eq geodesics}.  Bounded $r$-behaviour in $r<0$ is excluded by (\cite{KBH_book},  Lemma $4.14.2$).  $r_1=0$ is impossible because $\gamma$ will then crash into the ring singularity.  Hence,  $[r_1,r_2]\subset (0,r_{-})$ by Prop.  \ref{prop bounded have r2<r+},  which contradicts Prop.  \ref{no bounded geodesics with foliation}.


\subsubsection{Subcase $Q>0$,  cf.  \ref{case Q>0}}
\hfill \break

It is shown in \ref{case Q>0} that bounded $r$-behaviour is only possible in $r>0$ and that $R(0)<0$.  This implies that $[r_1,r_2]\subset (0,r_{-})$ by Prop.  \ref{prop bounded have r2<r+},  which contradicts Prop.  \ref{no bounded geodesics with foliation}.




\subsubsection{Subcase $Q<0$,  cf.  \ref{case Q<0}}
\hfill \break

Such null geodesics have constant $r$-behaviour at negative radius.  Two types of $\theta$-behaviour are possible.  If the $\theta$-coordinate is constant,  then the proof of Prop.  \ref{u-=u+ Q<0 prop} shows that such geodesics are unbounded because $t=t(s)$ must be a non-constant affine function.  If the $\theta$-coordinate is not constant,  the $\theta$-motion is periodic as shown in Prop.  \ref{prop theta behaviours Q<0} and the $t$-coordinate is a quasi-periodic function with a non-zero increment,  see \ref{case Q<0}.  Since the variation of the $t$-coordinate after $n$ $\theta$-oscillations is $n\Delta t$ (see Remark \ref{remark n oscillations}) and the proof of Thm.  \ref{main theorem} shows that $\Delta t>0$ for every negative fixed radius,  these geodesics cannot bounded.



\appendix
\addcontentsline{toc}{chapter}{Appendices}
\addtocontents{toc}{\setcounter{tocdepth}{-1}}

\section{Elliptic integrals and hypergeometric functions}\label{elliptic appendix}

\begin{defn}

Let $\phi\in[-\pi/2,\pi/2]$.  The elliptic integral of the first kind is

\begin{align*}
\mathcal{F}(\phi|k):=\int_0^{\sin\phi}\frac{ds}{\sqrt{(1-s^2)(1-ks^2)}}.
\end{align*}

The complete $(\phi=\pi/2)$ elliptic integral of the first kind is

\begin{align*}
\mathcal{K}(k):=\mathcal{F}(\pi/2|k)=\int_0^1\frac{ds}{\sqrt{(1-s^2)(1-ks^2)}}.
\end{align*}

The elliptic integral of the second kind is

\begin{align*}
\mathcal{E}(\phi|k):=\int_0^{\sin\phi}\sqrt{\frac{1-ks^2}{1-s^2}}ds.
\end{align*}

The complete $(\phi=\pi/2)$ elliptic integral of the second kind is

\begin{align*}
\mathcal{E}(k):=\mathcal{E}(\pi/2|k)=\int_0^1\sqrt{\frac{1-ks^2}{1-s^2}}ds.
\end{align*}

We define also
\begin{align*}
\mathcal{D}(k):=\int_0^1\frac{s^2ds}{\sqrt{(1-s^2)(1-ks^2)}}=\frac{\mathcal{K}(k)-\mathcal{E}(k)}{k}=-2\frac{\partial \mathcal{E}(k)}{\partial k}.
\end{align*}

\end{defn}

\begin{oss}\label{remark integral estimate}
Let $0<z,s<1, \; x<0$.  
\begin{align*}
\sqrt{\frac{1-zs^2}{1-s^2}}>\frac{1}{\sqrt{1-s^2}\sqrt{1-xs^2}}\hspace*{0.5cm}\Longleftrightarrow\hspace*{0.5cm}(1-zs^2)(1-xs^2)>1 \hspace*{0.5cm}\Longrightarrow\hspace*{0.5cm}\mathcal{E}(x)>\mathcal{K}(z).
\end{align*}
If $z=x/(x-1)$,  it satisfies $0<z<1$ and we have 
\begin{align*}
(1-zs^2)(1-xs^2)>1 \hspace*{0.5cm}\Longleftrightarrow\hspace*{0.5cm} x+z<xzs^2\hspace*{0.5cm}\Longleftrightarrow\hspace*{0.5cm} 1>s^2,
\end{align*}
hence $\mathcal{E}(x)>\mathcal{K}(z)$.
\end{oss}

\begin{defn}[\cite{hypergeo}]
The hypergeometric function $F(\alpha,\beta;\gamma;x)$ is defined by the series 

\begin{align*}
\sum_{n=0}^{\infty} \frac{(\alpha)_n (\beta)_n}{(\gamma)_n n!}x^n,
\end{align*}
where $(\alpha)_n:=\alpha(\alpha+1)\cdot ... \, \cdot (\alpha+n-1)$ for $n>0$,  $(\alpha)_0\equiv 1$ (analogous for the others),  for $|x|<1$,  and by continuation elsewhere.
\end{defn}

\begin{prop}[\textit{Euler's integral representation},  see \cite{hypergeo}]\label{Euler integral rep}
If $\textrm{Re}\;\gamma>\textrm{Re}\;\beta>0$,  then

\begin{align*}
F(\alpha,\beta;\gamma;x)=\frac{\Gamma(\gamma)}{\Gamma(\beta)\Gamma(\gamma-\beta)}\int_0^1 t^{\beta-1} (1-t)^{\gamma-\beta-1}(1-xt)^{-\alpha}dt
\end{align*}
in the complex $x-$plane cut along the real axis from $1$ to $+\infty$,  where $\Gamma(x):=\int_0^{\infty} t^{x-1} e^{-t} dt$ is the Euler's gamma function.
\end{prop}

\begin{prop}[\cite{hypergeo}]\label{elliptic as hypergeom}
We can write the complete elliptic integral of the first kind as

\begin{align*}
\mathcal{K}(x)=\frac{\pi}{2} F\bigg(\frac{1}{2},\frac{1}{2};1;x\bigg).
\end{align*}
\end{prop}

\begin{proof}
Use the integral representation of the hypergometric function given in Prop. \ref{Euler integral rep},  the integral substitution $t=s^2$,  with $\Gamma(\frac{1}{2})=\sqrt{\pi}$,  $\Gamma(1)=1$.
\end{proof}

\begin{prop}[\textit{"Pfaff's formula",  see Theorem $2.2.5$ of \cite{hypergeo}}]\label{hypergeometric trick}

\begin{align*}
F(\alpha,\beta;\gamma;x)=(1-x)^{-\alpha} F\bigg(\alpha, \gamma-\beta;\gamma;\frac{x}{x-1}\bigg).
\end{align*}
\end{prop}

\begin{proof}
Use \ref{Euler integral rep} and the integral substitution $t=1-s$.
\end{proof}

\section{Spherical null geodesics with $Q<0$}\label{appendix motivation}

\begin{prop}\label{prop. existence}
In the Kerr-star spacetime $K^*$,  null geodesics with constant radial coordinate and $\mathcal{Q}<0$ exist if and only if $r\in \big[ R_2(a,M),0\big)$,  where $R_2(a,M)$ is given by \eqref{roos of k(r) polynomial}.   
\end{prop}

\begin{oss}
Note that $R_2(a,M)\geq -M/2$,  so $B(r)<0$ by \eqref{B(r) expression} for all geodesics of this type.
\end{oss}

\begin{proof}[Proof of Prop.  \ref{prop. existence}]
($\Rightarrow$) Consider \eqref{theta di u}.  We must have $\tilde{\Theta}(u)\geq 0$,  hence $\textrm{dis}:=w^2+4a^2\mathcal{Q}\geq 0$,  where $w:=a^2-\Phi^2-\mathcal{Q}$.  We know that for spherical geodesics $\Phi$ and $\mathcal{Q}$ are given by \eqref{class r=const}.  Therefore we have

\begin{align}\label{equation discriminant}
\textrm{dis}=\frac{16Mr^2}{(M-r)^4}\Delta(r) (2r^3-3Mr^2+a^2M)=: \frac{16Mr^2}{(M-r)^4}\Delta(r) k(r)\geq 0.
\end{align}
Since $\Delta(r)>0$ for negative $r$,  the last inequality is equivalent to $k(r)\geq 0$.  The signs of the coefficients of $k(r)$ are $+\; -\; +$,  hence either there are two or zero positive roots.  But the sum of the roots is $3M/2>0$,  hence there are two positive roots.  Moreover $k(0)=a^2M>0$ and $\lim_{x\rightarrow -\infty}k(r)=-\infty$,  hence the third root is negative.  These roots can be expressed in the following way  (see \cite{zwillinger})    

\begin{align}\label{roos of k(r) polynomial}
R_j(a,M)=M\cos\bigg[ \frac{1}{3}\arccos\bigg(1-2\frac{a^2}{M^2}\bigg)-\frac{2}{3}j\pi \bigg]+\frac{M}{2}\hspace{1cm}j=0,1,2 ,
\end{align}
with $R_2(a,M)<0<R_1(a,M)<R_0(a,M)$.\\

\begin{figure}[H]
\centering
\includegraphics[scale=0.5]{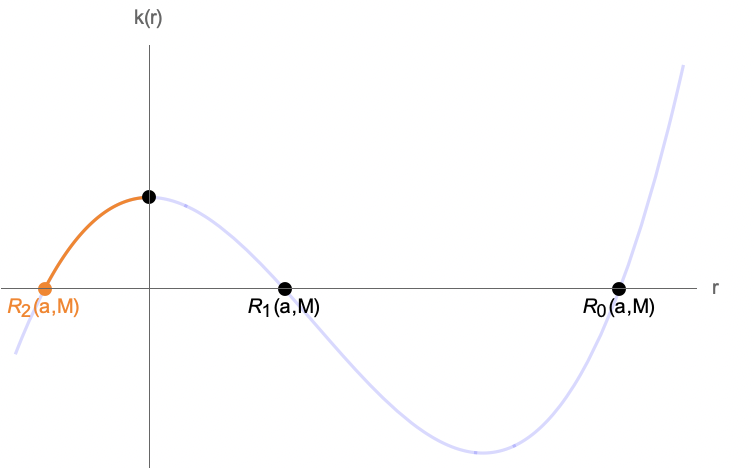} 
\caption{Plot of $k(r)$ for $a=3,M=5$.\\}
\vspace*{2mm}
\end{figure}
Hence \eqref{equation discriminant} can only be satisfied for $r\in \big[ R_2(a,M),0\big)$.
\vspace*{2mm}

\hspace{2mm}($\Leftarrow$) Assume first that $r=r_0\in \big( R_2(a,M),0\big)$ and fix $\Phi=\Phi(r_0)$,  $\mathcal{Q}=\mathcal{Q}(r_0)$ by \eqref{class r=const} so that 

\begin{align}\label{inequality discriminant with min carter}
\textrm{dis}=w^2(r_0)+4a^2\mathcal{Q}(r_0)=[\mathcal{Q}(r_0)+(|\Phi(r_0)|-|a|)^2][\mathcal{Q}(r_0)+(|\Phi(r_0)|+|a|)^2]>0,
\end{align}
with $w(r):=a^2-\Phi^2(r)-\mathcal{Q}(r)$.  Notice that $[\mathcal{Q}(r)+(|\Phi(r)|+|a|)^2]\geq [\mathcal{Q}(r)+(\Phi(r)-a)^2]=4r^2\Delta(r)/(M-r)^2>0$ since $\Delta(r)>0$ for $r<0$.  Therefore \eqref{inequality discriminant with min carter} implies that $[\mathcal{Q}(r_0)+(|\Phi(r_0)|-|a|)^2]>0$.  Since $r_0\in \big( R_2(a,M),0\big)$,  we have $\mathcal{Q}=\mathcal{Q}(r_0)<0$ by  \eqref{class r=const} and $w(r_0)>0$.  Indeed the inequality takes the form

\[
w(r_0)=a^2-\Phi^2(r_0)-\mathcal{Q}(r_0)=-2r_0\frac{r^3_0-3M^2r_0+2a^2M}{(M-r_0)^2}>0.
\]
This inequality is automatically satisfied for $r_0\in (R_2(a,M),0)$,  since $r_0^3-3M^2r_0\geq 0$ for $r_0\in[-\sqrt{3}M,0)$ and $-\sqrt{3}M<-M/2\leq R_2(a,M)$ by \eqref{roos of k(r) polynomial}.

Using \eqref{formula for u pm} we have then

\begin{align*}
0<u_{-}:=\frac{w(r_0)-\sqrt{w^2(r_0)+4a^2\mathcal{Q}(r_0)}}{2a^2}<\frac{w(r_0)}{2a^2}<\frac{a^2-\Phi(r_0)^2+(|\Phi(r_0)|-|a|)^2}{2a^2}=1-\frac{|\Phi(r_0)|}{|a|}\leq 1,
\end{align*}
so that we can define $\theta_2:=\arccos(\sqrt{u_{-}})\in (0,\pi)$.  We now fix the initial point $p_0=(0,r_0,\theta_0,0)\in K^*$ with $\theta_0:=\theta_2$,  $r_0\in(R_2(a,M),0)$ and the following set of constants of motion $(q=0,E=1,L=\Phi(r_0),  K=\mathcal{Q}(r_0)+(\Phi(r_0)-a)^2 )$ by \eqref{class r=const}.  Then $\Theta(\theta_0)=0$ by \eqref{theta di u} and $R(r_0)=0$ by \eqref{class r=const}.

By Prop.  $4.2.6$ in \cite{KBH_book},  there exists a null geodesic starting at $p_0$ with that particular set of constants of motion.

Assume now $r=R_2(a,M)$.  Choose a converging sequence of initial points $p_n=(0,r_n,\theta_n,0)\in K^*$ with $(R_2(a,M),0)\ni r_n\rightarrow R_2(a,M)$,  $\theta_n:=\arccos(\sqrt{u_{-}(r_n)})\in (0,\pi)$ with 

\begin{align*}
u_{-}(r):=\frac{w(r)-\sqrt{w^2(r)+4a^2\mathcal{Q}(r)}}{2a^2},
\end{align*}
and consider the corresponding geodesic $\gamma_n$ constructed above. The constants of motion of $\gamma_n$ converge and hence so do the tangent vectors $\gamma'_n(p_n)$ by \eqref{geodes diff equations}.  By continuity,  there exists a geodesic $\gamma=\lim\gamma_n$ with $\mathcal{Q}=\mathcal{Q}[R_2(a,M)]<0$ by \eqref{class r=const},  starting at the point $p_0=(0,R_2(a,M),\arccos\bigg( \sqrt{u_{-}[R_2(a,M)]}\bigg),0)$ with $(q=0,E=1,L=\Phi[R_2(a,M)],  K=\mathcal{Q}[R_2(a,M)]+[\Phi(R_2(a,M))-a]^2 )$ as constants of motion.  The limit geodesic has constant $\theta$-coordinate and falls in the class described in Prop.  \ref{u-=u+ Q<0 prop}.
\end{proof}

\section*{Declarations}

\subsection*{Financial or Non-financial Intersts} The author has no relevant financial or non-financial interests to disclose.

\subsection*{Conflict of interest} The author has no competing interests to declare that are relevant to the content of this article.

\addcontentsline{toc}{chapter}{Conclusions}

\bibliography{bibliography}
\bibliographystyle{acm}

\end{document}